\documentclass[12pt,leqno]{amsart}
\usepackage{color,amsmath,amsfonts,amssymb,amscd,amsthm,amsbsy,upref, siunitx}
\numberwithin{equation}{section}
\def\hangbox to #1 #2{\vskip3pt\hangindent #1\noindent \hbox to #1{#2}$\!\!$}

\newtheorem{thm}{Theorem}[section]

\newtheorem{lem}[thm]{Lemma}
\newtheorem{cor}[thm]{Corollary}

\theoremstyle{definition}

\newtheorem{claim}[thm]{Claim}
\newtheorem{defn}[thm]{Definition}
\newtheorem{defin}[thm]{Definition}
\newtheorem{problem}[thm]{Problem}

\theoremstyle{remark}

\def\C{{\mathbb C}}
\def\N{{\mathbb N}}
\def\R{{\mathbb R}}

\def\Z{{\mathbb Z}}

\newcommand{\ra}{\rangle}
\newcommand{\la}{\langle}

\def\sfrac#1#2{\kern.1em\raise.5ex\hbox{$#1$}
        \kern-.1em/\kern-.05em\lower.25ex\hbox{$#2$}}

\def\vp{\varepsilon}
\def\dim{\operatorname{dim}}

\newcommand{\fw}{\text{\fw}}

\begin{document}
\allowdisplaybreaks
\title{The discretization problem for continuous frames}

\author{ Daniel Freeman}
\address{Department of Mathematics and Statistics\\
St Louis University\\
St Louis, MO 63103  USA} \email{dfreema7@slu.edu}

\author{Darrin Speegle}
\address{Department of Mathematics and Statistics\\
St Louis University\\
St Louis, MO 63103  USA} \email{speegled@slu.edu}

\thanks{
The first author was supported by grant 353293 from the Simon's foundation, and the second author  was supported by grant 244953 from the Simon's foundation.}


\begin{abstract}
 We characterize when a coherent state or continuous frame for a Hilbert space may be sampled to obtain a frame, which solves the discretization problem for continuous frames.  In particular, we
prove that every bounded continuous frame for a Hilbert space may be sampled to obtain a frame.  We give multiple applications to different classes of frames such as scalable frames and Gabor frames.
\end{abstract}


\maketitle


\section{Introduction}

\begin{defin}
A collection of vectors $(x_{j})_{j\in J}$ in a Hilbert space $H$ is called a {\em frame} or a {\em discrete frame} for $H$ if there exists positive constants $A$ and $B$ (called lower and upper {\em frame bounds} respectively) such that
\begin{equation}\label{e:frame}
A\|x\|^2\leq \sum_{j\in J} |\langle x, x_i\rangle|^2 \leq
B\|x\|^2\quad\textrm{for all }x\in H.
\end{equation}
The frame is called {\em tight} if $A=B$ and the frame is called {\em
Parseval} if $A=B=1$.
\end{defin}

A frame can be thought of as a (possibly) redundant coordinate system in the sense that a frame can contain more vectors than are necessary to represent each vector in the Hilbert space.
One way of
 interpreting the frame inequality \eqref{e:frame} is  that a frame for a Hilbert space $H$ is a collection of vectors in $H$ indexed by a countable set $J$ so that the norm in $\ell_2(J)$ of the frame coefficients is equivalent to the norm on $H$.  This notion can be nicely generalized from the discrete to the continuous setting by instead of summing over a countable set $J$, we integrate over a measure space $X$.  That is, a continuous frame for a Hilbert space $H$ is a collection of vectors indexed by a measure space $X$ so that the norm of the frame coefficients in $L_2(X)$ is equivalent to the norm on $H$.  The following definition, with which we are following \cite{C}, formalizes this notion.

\begin{defn}
Let $(X, \Sigma, \mu)$ be a positive, $\sigma$-finite measure space and let $H$ be a separable Hilbert space.  A measurable function $\Psi:X\rightarrow H$ is a {\em continuous frame} of $H$ with respect to $\mu$ if there exists constants $A,B>0$ such that
\begin{equation}
A\|x\|^2\leq \int |\langle x, \Psi(t)\rangle|^2 d\mu(t)\leq B \|x\|^2\quad\quad\forall x\in H.
\end{equation}
The constant $A$ is called a {\em lower frame bound} and the constant $B$ is called an {\em upper frame bound}.  If $A=B$ then the continuous frame is called {\em tight} and if $A=B=1$ then the continuous frame is called {\em Parseval} or a {\em coherent state}.
We say that $\Psi:X\rightarrow H$ is {\em Bessel} if it has a finite upper frame bound $B$, but does not necessarily have a positive lower frame bound $A$.
\end{defn}

Note that if $X$ is a countable set with counting measure, then $\Psi:X\rightarrow H$ is a continuous frame of $H$  is equivalent to $(\Psi(t))_{t\in X}$ being a frame of $H$.  Thus, frames are a special case of continuous frames.      Often the definition of continuous frames in the literature includes some additional topological structure of the measure space $X$ and continuity of the map $\Psi:X\rightarrow H$, but we do not require this.

 Continuous frames and coherent states are widely used in mathematical physics and are particularly prominent in quantum mechanics and quantum optics. The theory of coherent states was initiated by Schr\"odinger in 1926 \cite{S} and was generalized to continuous frames by Ali, Antoine, and Gazeau \cite{AAG1}.  Though
 coherent states naturally  characterize many different physical properties,  discrete frames are much better suited for  computations.  Because of this, when working with coherent states and continuous frames, researchers often create a discrete frame by sampling the continuous frame and then use the discrete frame for computations instead of the entire continuous frame.
Specifically, if $\Psi:X\rightarrow H$ is a continuous frame and $(t_j)_{j\in J}\subseteq X$ then $(\Psi(t_j))_{j\in J}\subseteq H$ is called a {\em sampling} of $\Psi$.   The notion of creating a frame by sampling a coherent state has its origins in the very start of modern frame theory.
Indeed, Daubechies, Grossmann, and Meyer \cite{DGM} popularized modern frame theory in their seminal paper ``Painless nonorthogonal expansions", and their constructions of frames for Hilbert spaces were all done by sampling different coherent states.  Another example which is of particular interest for frame theorists is that of Gabor frames, which are samplings of the short time Fourier transform at a lattice, and we explore this further in Section \ref{S:app}.

The discretization problem, posed by Ali, Antoine, and Gazeau in their physics textbook {\em Coherent States, Wavelets, and Their Generalizations} \cite{AAG2},  asks when a continuous frame of a Hilbert space can be sampled to obtain a frame.  They state that a positive answer to the question is crucial for practical applications of coherent states, and chapter 16 of the book is devoted to the discretization problem.
A solution for certain types of continuous frames was obtained by Fornasier and Rauhut using the theory of co-orbit spaces \cite{FR}.
We solve the discretization problem in its full generality  with the following theorem which characterizes exactly when a continuous frame may be sampled to obtain a frame.

\begin{thm}\label{T:D}
Let $(X,\Sigma)$ be a measure space such that every singleton is measurable and let $\Psi:X\rightarrow H$ be measurable.  There exists $(t_j)_{j\in J}\in X^J$ such that $(\Psi(t_j))_{j \in J}$
 is a frame of $H$ if and only if there exists a positive, $\sigma$-finite measure $\nu$ on $(X,\Sigma)$ so that  $\Psi$ is a continuous frame of $H$ with respect to $\nu$ which is bounded $\nu$-almost everywhere.  In particular, every bounded continuous frame may be sampled to obtain a frame.
\end{thm}

Here we mean that a continuous frame $\Psi:X\rightarrow H$ is {\em bounded $\nu$-almost everywhere} if there exists a constant $C>0$ and measurable subset $E\subseteq X$ with $\nu(E)=0$ such that $\|\Psi(t)\|\leq C$ for all $t\in X\setminus E$.

Continuous frames used in applications are typically bounded. However, there do exist examples of unbounded continuous frames which cannot be sampled to obtain a frame.  For example, consider the measure $\mu$ on $\N$ given by $\mu(\{n\})=1/n$.  Then $\Psi:\N\rightarrow\ell_2$ with $\Psi(n)=\sqrt{n} e_n$ is an unbounded continuous frame, and any sampling of $\Psi$ with dense span is unbounded and hence not a frame.  We note as well that our sampling $(\Psi(t_j))_{j\in J}$ in the statement of Theorem \ref{T:D} allows for some points to be sampled multiple times, and indeed there exist bounded continuous frames such that the only way for a sampling to be a frame is for some points to be sampled multiple times.

Theorem \ref{T:D} characterizes when a continuous frame may be sampled to obtain a frame.  The following theorem gives uniform frame bounds for the sampled frame in the case that the continuous frame is Parseval and in the unit ball of the Hilbert space.  We first prove this uniform theorem in Section \ref{S:disc} and then use this to prove Theorem \ref{T:D}.  

\begin{thm}\label{T:UnifSamp}
There exists $C,D>0$ such that if $\Psi:X\rightarrow H$ is a continuous Parseval frame with $\|\Psi(t)\|\leq 1$ for all $t\in X$ then  there exists $(t_j)_{j\in J}\in X^J$ such that $(\Psi(t_j))_{j \in J}$ is a frame of $H$ with lower frame bound $C$ and upper frame bound $D$.
\end{thm}

So far we have stated results  solely in terms of sampling continuous frames.  However, as every frame can be realized as a continuous frame on $\N$, we may consider what the discretization problem implies in this case.  Suppose that $(x_n)_{n\in\N}$ is a tight frame of a Hilbert space $H$ with frame bound $K>1$ satisfying $\|x_n\|\leq 1$ for all $n\in\N$.  If we define a measure $\mu$ on $\N$ with $\mu(n)=1/K$ for all $n\in\N$, then $\Psi:\N\rightarrow H$ with $\Psi(n)=x_n$ is a continuous Parseval frame.  Now, sampling $\Psi$ corresponds with taking a subset of $(x_n)_{n\in\N}$.   This gives the following corollary, which was proven for finite frames by Nitzan, Olevskii, and Ulanovskii \cite{NOU}.

\begin{cor}\label{C:frameSubset}
There exists uniform constants $E,F>0$ such that if $(x_i)_{i\in I}$ is a tight frame of vectors in the unit ball of a Hilbert space $H$ with frame bound greater than $1$ then there exists a subset $J\subseteq I$ such that $(x_i)_{i\in J}$ is a frame with lower frame bound $E$ and upper frame bound $F$.
\end{cor}

For applications, this means that if you are working with a frame which is more redundant than necessary then it is possible to a take a subset which is less redundant and is still a good frame.  In Section \ref{S:part} we will give a direct proof of a generalization of Corollary \ref{C:frameSubset} which we state as Corollary \ref{C:InfPart} later in the introduction.
To realize a more general application of this idea, we call a collection of vectors $(x_i)_{i \in I}$ in a Hilbert space $H$  a scalable frame  if there exists scalars $(c_i)_{i\in I}$ such that $(c_ix_i)_{i\in I}$ is a Parseval frame for $H$ \cite{KOPT}.   We can use this to define a measure $\mu$ on $I$ by $\mu(i)=|c_i|^2$ for all $i\in I$.  Then,  $\Psi: I\rightarrow H$ with $\Psi(i)=x_i$ is a continuous Parseval frame and applying Theorem \ref{T:UnifSamp} gives the following corollary.

\begin{cor}\label{C:frameScalable}
There exists uniform constants $E,F>0$ such that if $(x_i)_{i\in I}$ is a scalable frame of unit vectors then there exists a subset $J\subseteq I$ such that $(x_i)_{i\in J}$ is a frame with lower frame bound $E$ and upper frame bound $F$.
\end{cor}
We give further discussion of scalable frames in Section \ref{S:app}, where we prove a new quantization theorem for scalable frames.

Our approach to proving Theorem \ref{T:D} in Section \ref{S:disc} is based on reducing the problem of sampling continuous frames into a different problem of partitioning discrete frames, which is solved using the recent results of Marcus,  Spielman, and Srivastava in their solution of the Kadison-Singer problem \cite{MSS}.
  In \cite{BCEK}, Balan, Casazza, Edidin, and Kutyniok introduce problems of determining when a tight frame for a Hilbert space may be partitioned into subsets each of which are frames for the Hilbert space and what are the optimal such decompositions.
The notion of an optimal decomposition depends on the application, but one natural way to consider a partition of a tight frame into subsets as optimal is if each of the lower frame bounds of the subsets are as big as possible and each of the upper frame bounds of the subsets are as small as possible.   In that respect, the following theorem, which is essentially Lemma 2 in Nitzan, Olevskii, and Ulanovskii's paper \cite{NOU}, proves that there exists uniform upper and lower frame bounds for the optimal decompositions of tight frames.    They used this result to then prove that for every subset $S\subset \R$ of finite Lebesgue measure, there exists a discrete set $\Lambda\subset\R$ such that the exponentials $(e^{i\lambda t})_{\lambda\in\Lambda}$ form a frame of $L_2(S)$.  
We include a proof of Theorem \ref{T:reduction1} in Section \ref{S:part} for completeness as we show that it is essentially the uniform discretization problem for continuous frames with finite support.

\begin{thm}\label{T:reduction1}
There exists uniform constants $A,B>0$ such that every tight frame of vectors in the unit ball of a finite dimensional Hilbert space $H$ with frame bound greater than $1$ can be partitioned into a collection of frames of $H$ each with lower frame bound $A$ and upper frame bound $B$.  Furthermore, there exists sequences of constants $(A_n)_{n\in\N},(B_n)_{n\in\N}$ with $\lim_{n\rightarrow\infty}A_n=\lim_{n\rightarrow\infty}B_n=1$ such that  for every $N\in\N$ and every tight frame of vectors in $N^{-1}B_H$ with frame bound greater than $1$ has a subset with lower frame bound $A_N$ and upper frame bound $B_N$. 
\end{thm}

At first glance, it may seem obvious that a highly redundant frame is the union of less redundant frames, but in a high dimensional Hilbert space it can be very tricky to determine how to partition the frame vectors so that each set in the resulting partition has a uniformly high lower frame bound and uniformly low upper frame bound.  Furthermore, we prove that the corresponding question for bases is false.  That is, there does not exist a uniform constant $A$ such that every finite unit norm tight frame has a subset which is a basis and has lower Riesz bound $A$.  Using this finite dimensional result, we prove the following generalization at the end of Section \ref{S:part}.

\begin{cor}\label{C:InfPart}
Let $A,B>0$ be the constants given in Theorem \ref{T:reduction1}. Then every tight frame of vectors in the unit ball of a separable Hilbert space $H$ with frame bound greater than $1$ can be partitioned into a finite collection of frames of $H$ each with lower frame bound $A$ and upper frame bound $B$. 
\end{cor}

The proof of Theorem \ref{T:reduction1} relies on the recent solution of the Kadison-Singer problem by Marcus, Spielman, and Srivastava \cite{MSS}.  The Kadison-Singer problem \cite{KS} was known to be equivalent to many open problems such as the Feichtinger conjecture \cite{CCLV}, the paving conjecture \cite{A}, Weaver's conjecture \cite{W}, and the Bourgain-Tzafriri conjecture \cite{BT}.  Each of these problems can be thought of in some ways as determining when a set with some property can be uniformly partitioned into sets with a desired property.
Naturally, the frame partition problem falls into this category as well, and it was noted in \cite{BCEK} that the problem of partitioning a large frame into smaller frames is related to these famous problems.  In \cite{MSS}, the authors directly prove Weaver's conjecture and hence prove all the equivalent problems as well.  We note that the proof of Theorem \ref{T:reduction1} doesn't actually use Weaver's conjecture or any of its equivalent formulations, but instead uses the stronger result proved in \cite{MSS}.  The main reason for this is that Weaver's conjecture concerns partitioning in a way that reduces an upper bound, but we need to reduce an upper bound while maintaining a relatively  close lower bound.  The Marcus, Spielman, Srivastava result allows partitioning in a way that divides both the upper and lower frame bound almost perfectly in half.

We recommend the textbook \cite{C} for a reference on frames and continuous frames from a mathematical perspective, and we recommend the textbook \cite{AAG2} for a reference on frames and continuous frames from a physics perspective.   We include applications of the discretization and partitioning theorems in Section \ref{S:app}. We prove the frame partitioning theorem in Section \ref{S:part}. We include lemmas about discrete frames in Section \ref{S:lem}.  We prove the discretization theorem in Section \ref{S:disc}.

\section{Applications}\label{S:app}
We include here some applications of the discretization and partition theorems, two of which are new theorems and one is a quick proof of a known theorem.

\subsection{Scalable Frames}\label{S:SF}

A collection of unit vectors $(x_i)_{i \in I}$ in $H$ is said to be a scalable frame  if there exist scalars $(c_i)_{i\in I}$ such that $(c_ix_i)_{i\in I}$ is a Parseval frame for $H$ \cite{KOPT}. 
As one of the steps toward solving the discretization problem, we prove  in Theorem \ref{T:samp} that there are universal constants $A$ and $B$ such that if $(x_i)_{i\in I}$ is a scalable frame, then $(x_i)_{i\in I}$ can be sampled to form a frame with lower frame bound $A$ and upper frame bound $B$.  One way to think about this is in terms of quantization.  A scalable frame can be scaled to be Parseval, but suppose that we are restricted to using only integer coefficients to scale the frame.  Being able to sample $(x_i)_{i \in I}$ to get a frame is equivalent to being able to obtain a frame by scaling using only integer scalars.  As quantization of frame coefficients is an important aspect of frame theory \cite{BPY}, it is natural to consider scaling frames using quantized scalars.  The following result gives essentially that how well we can scale a frame using quantized coefficients may be determined using only how fine a quantization we allow.  In particular, this is independent of both the dimension of the space and the number of frame vectors.

\begin{thm}\label{T:QS}
Let $(A_n)_{n\in\N}$ and $(B_n)_{n\in\N}$ with $\lim A_n=\lim B_n=1$ be the scalars given in Theorem \ref{T:reduction1}.  Let $N\in\N$.  If $(x_i)_{i \in I}$ is a scalable frame in a finite dimensional Hilbert space $H$
then there exists scalars $(c_i)_{i\in I}\subseteq\{ \sqrt{m}/N:m\in\Z\}$ such that $(c_i x_i)_{i\in I}$ is a frame with lower frame bound $A_N$ and upper frame bound $B_N$.  
\end{thm}
\begin{proof}
Let $(d_i)_{i\in I}$ be scalars such that $(d_i x_i)_{i\in I}$ is a Parseval frame.  We first assume that $(d_i)_{i\in I}$ are rational numbers with common denominator $M\in\N$.  Thus, $n_i:=N^2 M^2 d^2_i$ is an integer for all $i\in I$.     We consider the frame $(y_j)_{j\in J}$ which consists of $n_i$ copies of $N^{-1} x_i$ for each $i\in I$.  Thus, we have for $x\in H$ that
$$\sum_{j\in J} |\langle x, y_j\rangle|^2=\sum_{i\in I} N^2 M^2 d_i^2   |\langle x, N^{-1}x_i\rangle|^2=M^2\sum_{i\in I}  |\langle x, d_ix_i\rangle|^2=M^2\|x\|^2.
$$
Thus, $(y_j)_{j\in J}$ is a tight frame of vectors in $N^{-1} B_H$ with frame bound $M^2\geq 1$.  By Theorem \ref{T:reduction1} there is a subset $(y_j)_{j\in J_0}$ with frame bounds $A_N$ and $B_N$.  For each $i\in I$ let $c_i=\sqrt{m_i}/N$ where $m_i$ is the number of copies of $N^{-1} x_i$ in  $(y_j)_{j\in J_0}$.  This gives the following calculation.
$$\sum_{i\in I} |\langle x, c_i x_i\rangle|^2=\sum_{i\in I} m_i|\langle x, N^{-1}x_i\rangle|^2=\sum_{j\in J_0}  |\langle x, y_j\rangle|^2.$$
Thus, $(c_i x_i)_{i\in I}$ is a frame with bounds $A_N, B_N$, the same bounds as $(y_j)_{j\in J_0}$.

We now consider the case that $(d_i)_{i\in I}$ are not all rational.  Given $\vp>0$ we may approximate the coefficients $(d_i)_{i\in I}$ with rational numbers and follow the previous argument to obtain $(c_i)_{i\in I}\subseteq\{ \sqrt{m}/N:m\in\Z\}$ such that $(c_i x_i)_{i\in I}$ is a frame with lower frame bound $(1+\vp)^{-1}A_N$ and upper frame bound $(1+\vp)B_N$.  However, because there are only finitely many possibilities for $(c_i)_{i\in I}$, some choice must work for all $\vp>0$.  Thus, there exists $(c_i)_{i\in I}\subseteq\{ \sqrt{m}/N:m\in\Z\}$ such that $(c_i x_i)_{i\in I}$ is a frame with lower frame bound $A_N$ and upper frame bound $B_N$.

\end{proof}

Theorem \ref{T:QS} gives that if there are scalars $(d_i)_{i\in I}$ such that $(d_i x_i)_{i \in I}$ is Parseval then there are scalars $(c_i)_{i\in I}\subseteq\{ \sqrt{m}/N:m\in\Z\}$ such that $(c_i x_i)_{i\in I}$ is a frame with lower frame bound $A_N$ and upper frame bound $B_N$.  It is interesting to note that the scalars $(c_i)_{i\in I}$ could be very different from the scalars $(d_i)_{i\in I}$.  Indeed, quantizing $(d_i)_{i\in I}$ fails dramatically if we choose $(c_i)_{i\in I}$ to minimize $|d_i-c_i|$ for all $i\in I$, which results in a small $\ell_\infty(I)$ distance between the sequences $(d_i)_{i\in I}$ and $(c_i)_{i\in I}$.  A small $\ell_2(I)$ distance between $(d_i)_{i\in I}$ and $(c_i)_{i\in I}$ can be used to compare the frames $(d_i x_i)_{i \in I}$ and $(c_i x_i)_{i \in I}$ \cite{C2}, but a small $\ell_\infty(I)$ distance tells us nothing.

\subsection{Gabor Frames}\label{S:GF}

 One example of continuous frames that is of particular interest is that of Gabor systems and the short time Fourier transform.  We consider $L^2(\R)$ to be the Hilbert space of square integrable functions from $\R$ to $\C$.
For $a, b \in \R$, we define the translation operator $T_a:L^2(\R) \to L^2(\R)$ and the modulation operator $M_b:L^2(\R) \to L^2(\R)$ by
\[
T_a g(x) = g(x - a) \qquad {\mathrm {and}} \qquad M_b g (x) = e^{2\pi i b x} g(x).
\]
If $g\in L^2(\R)$ then {\em the short time Fourier transform} with window function $g$  is the map $\Psi_g : \R^2 \to L^2(\R)$ given by 
\[
\Psi_g(a,b) = M_b T_a g.
\]
The short time Fourier transform with window function $g$ is a tight continuous frame with frame bound $\|g\|^2$.  That is, $\|g\|^2 \|f\|^2=\int \int |\int \overline{f(x)} e^{2\pi i \omega x}g(x- t) dx|^2 d\omega dt$ for all $f\in L_2(\R)$.
A {\em Gabor frame} of $L_2(\R)$ is a frame of the form $(M_{bn} T_{am} g)_{m,n\in\Z}$ where  $a,b>0$ and $g\in L^2(\R)$.  That is, Gabor frames are formed by sampling the short time Fourier transform at a lattice in $\R^2$.  It is not always the case that  $(M_{bn} T_{am} g)_{m,n\in\Z}$ will be a frame, however by Theorem \ref{T:D} we have the following corollary.
\begin{cor}\label{C:GF}
 For every non-zero $g\in L^2(\R)$ there exists real numbers $(a_k, b_k)_{k=1}^\infty$ such that $(e^{2\pi i b_k x} g(x - a_k))_{k\in\N}$
is a frame of $L^2(\R)$. 
\end{cor}

\subsection{Frames of exponentials}\label{S:SF}
For each $K>0$, Fourier series gives a Riesz basis of exponentials for $L_2([-\frac{K}{2},\frac{K}{2}])$.  In particular,
$$ K f = \sum_{n\in\Z} \langle f, e^{2\pi i \frac{n}{K}\cdot}\rangle e^{2\pi i \frac{n}{K} \cdot}  \quad\quad \textrm{ for all }f\in L_2([-\frac{K}{2},\frac{K}{2}]).
$$

If $J\subseteq \R$ is any bounded set, then $J\subseteq I$ where $I$ is an interval.  We can take the basis of exponentials for $L_2(I)$ and restrict it to $J$ to get a tight frame of exponentials for 
$L_2(J)$.  However, this does not work if $J$ is unbounded.   This leads to the question: When does $L_2(J)$ have a frame of exponentials? Note  that $J$ must have finite measure for the exponentials to be in $L_2(J)$.  This was solved by Nitzan, Olevskii, and Ulanoskii, who used the same frame partitioning theorem that we need as well \cite{NOU}.

\begin{thm}\label{T:NOU}[Nitzan, Olevskii, and Ulanovskii 2016]
If $J\subseteq\R$ has finite measure then $L_2(J)$ has a frame of exponentials.
\end{thm}

Unlike the short time Fourier transform,  the Fourier transform is not a continuous frame. However, if we consider the map $\Psi: \R\rightarrow L_2(J)$ given by $\Psi(x)(t)=e^{2\pi i x t}$ for all $x\in \R$  and $t\in J$ then $\Psi$ has an analysis operator $\Theta:L_2(J)\rightarrow L_2(\R)$ given by 
$$\Theta(f)(x)=\langle f, \Psi(x)\rangle=\int_{t\in J} f(t) e^{2\pi i x t} d\,\lambda(t).
$$
Thus, the analysis operator is the Fourier transform.  The Fourier transform is an isometric embedding, which means that $\Psi$ is a continuous Parseval frame. Hence, by the discretization theorem, $\Psi$ may be sampled to give a frame of exponentials for $L_2(J)$ which gives Theorem \ref{T:NOU} as a corollary.

\section{Frame Partitions}\label{S:part}

Our goal for this section is to prove Theorem \ref{T:reduction1}, on uniformly partitioning frames.  
The main ingredient of the proof is the following theorem of Marcus, Spielman, and Srivistava.

\begin{thm}[MSS Cor 1.5] \label{MSS}  Let $(u_i)^M_{i=1} \subseteq H$ be a Bessel sequence with bound $1$ and $\|u_i\|^2 \leq \delta$ for all $i$. Then for any positive integer $r$, there exists a partition $\{I_1, . . . , I_r\}$ of $[M]$ such that each $(u_i)_{i\in I_j}$ , $ j = 1, . . . , r$  is a Bessel sequence with bound
$$(1/\sqrt{r}+\sqrt{\delta})^2
$$
\end{thm}

We will be applying Theorem \ref{MSS} for $r=2$ to partition a Parseval frame into two sets with Bessel bound close to $1/2$.  Theorem \ref{MSS} gives good control of the upper frame bound when partitioning a Parseval frame, but we need to control the lower frame bound as well.  The following theorem can be applied to show that if a Parseval frame is partitioned into two sets with upper frame bound close to $1/2$ then the sets also have lower frame bound close to $1/2$.

\begin{thm}[BCMS Cor 4.6]\label{46} Let $P:\ell^2(I)\rightarrow\ell^2(I)$ be orthogonal projection onto a closed subspace $H\subseteq \ell^2(I)$.  Then for any subset $J\subset I$ and $\delta>0$, TFAE
\begin{enumerate}
\item $\{P e_i\}_{i\in J}$ is a frame of $H$ with frame bounds $\delta$ and $1-\delta$.
\item $\{P e_i\}_{i\in J^c}$ is a frame of $H$ with frame bounds $\delta$ and $1-\delta$.
\item Both $\{P e_i\}_{i\in J}$ and $\{P e_i\}_{i\in J^c}$ are Bessel with bounds $1-\delta$.
\item Both $\{(I-P) e_i\}_{i\in J}$ and $\{(I-P) e_i\}_{i\in J^c}$ are Riesz sequences with lower bound $\delta$.
\end{enumerate}

\end{thm}

We will be repeatedly partitioning a frame using Theorem \ref{MSS} and then  applying a positive self-adjoint invertible operator to the resulting sets.  The following simple lemma allows us to keep track of what the operators do to the frame bounds.

\begin{lem}\label{L:operator}
Let $(x_j)_{j\in J}$ be a Parseval frame of a Hilbert space $H$. Let
$T$ be a positive  self adjoint invertible operator on $H$.  Then
$(Tx_j)_{j\in J}$ is a frame of $H$ with upper frame bound $\|T\|^2$
and lower frame bound $\|T^{-1}\|^{-2}$
\end{lem}
\begin{proof}
Let $x\in H$.  To calculate the upper frame bound we have the
following inequalities.
$$\sum_{j\in J}|\langle x,Tx_j\rangle|^2=\sum_{j\in J}|\langle
Tx,x_j\rangle|^2 =\|Tx\|^2\leq\|T\|^2\|x\|^2.
$$
Thus $\|T\|^2$ is the upper frame bound.  To calculate the lower
frame bound we have the following inequalities.
$$\|T^{-1}\|^{-2}\|x\|^2\leq \|Tx\|^2= \sum_{j\in J}|\langle Tx,x_j\rangle|^2=\sum_{j\in J}|\langle
x,Tx_j\rangle|^2.
$$
Thus, $\|T^{-1}\|^{-2}$ is the lower frame bound.

\end{proof}

We are now ready to prove the main result of this section.  The first part appears essentially as Lemma 2 in  \cite{NOU}, but we include a proof here for completion as it is essentially the uniform discretization problem for continuous frames with finite support.

\begin{thm}\label{T:reduction}
There exists uniform constants $A,B>0$ such that every tight frame of vectors in the unit ball of a finite dimensional Hilbert space $H$ with frame bound greater than $1$ can be partitioned into a collection of frames of $H$ each with lower frame bound $A$ and upper frame bound $B$.  Moreover, there exists sequences of constants $(A_n)_{n\in\N},(B_n)_{n\in\N}$ with $\lim_{n\rightarrow\infty}A_n=\lim_{n\rightarrow\infty}B_n=1$ such that  for every $N\in\N$ and every tight frame of vectors in $N^{-1}B_H$ with frame bound greater than $1$ has a subset with lower frame bound $A_N$ and upper frame bound $B_N$. 
\end{thm}

\begin{proof}
We prove the first claim of the theorem and then discuss at the end how the proof could be adapted to prove the moreover claim.
 For convenience, we will only consider tight frames with frame bound at least $79$.  Then we will find a uniform constant $B> 79$ so that every such tight frame can be partitioned into frames with upper frame bound $B$ and lower frame bound $79$.  Thus, any tight frame in the unit ball of a finite dimensional Hilbert space with frame bound greater than $1$ could be partitioned into a set of frames of $H$ with upper frame bound $B$ and  lower frame bound $1$. 

The proof will involve repeated application of Theorem \ref{MSS} so that at each step we will partition a frame into two frames with the same upper frame bound and same lower frame bound.  We will then choose one of those frames to partition further until we arrive at a set which is close to being tight and has small upper frame bound.  As we could do the same procedure to the frames not chosen, we are able to partition our original frame into frames which are close to being tight and have small upper frame bound.

 Assume that $(x_j)_{j\in J_0}$ is a tight frame in the unit ball of $H$ with frame bound  $B_0\geq 79$.  
 We recursively define a decreasing sequence $B_0>B_1>\cdots > B_n$ by $B_{m+1}= 2^{-1}B_m-2^{1/2}B_m^{1/2}-1$ for all $1\leq m< n$
where $n\in\N_0$ is such that $200>B_n\geq 79$.  Note that $79=
2^{-1}200-2^{1/2}200^{1/2}-1$ and thus there is a unique $n\in\N_0$ such that $200>B_n\geq 79$.
 We will choose by induction a nested sequence of subsets
$J_0\supseteq J_1\supseteq... \supseteq J_n$  so that if $1\leq
m\leq n$, $T_0$ is the identity, and $T_m$ is the frame operator of
$(T^{-1/2}_{m-1}...T^{-1/2}_1 B_0^{-1/2} x_j)_{j\in J_{m}}$ then
\begin{equation}\label{E:1}
\|T^{-1/2}_{m}...T^{-1/2}_0 B_0^{-1/2}\|^2\leq B_m^{-1}\quad\textrm{ for }0\leq m\leq n,
\end{equation}
\begin{equation}\label{E:2}
\|T_{m}\|\leq 2^{-1}+ 2^{1/2}B_m^{-1/2}+B_m^{-1} \quad\textrm{ for }1\leq m\leq n,
\end{equation}
\begin{equation}\label{E:3}
\|T_m^{-1}\|\leq (2^{-1}- 2^{1/2}B_m^{-1/2}-B_m^{-1})^{-1} \quad\textrm { for }1\leq m\leq n.
\end{equation}

For the base case $m=0$ we have that  \eqref{E:1},  \eqref{E:2} and \eqref{E:3} are all trivially satisfied.

 Let $0\leq m<n$ and assume that $J_0\supseteq\cdots \supseteq J_m$ have been chosen to satisfy \eqref{E:1}, \eqref{E:2}, and \eqref{E:3}.
 As $T_{m}$ is the frame operator of $(T^{-1/2}_{m-1}...T^{-1/2}_1 B_0^{-1/2} x_j)_{j\in J_{m}}$ we have that $(T^{-1/2}_{m}...T^{-1/2}_1 B_0^{-1/2} x_j)_{j\in J_{m}}$ is a Parseval frame.  Furthermore, $\|T^{-1/2}_{m}...T^{-1/2}_1 B_0^{-1/2} x_j\|\leq B^{-1}_m$ for all $j\in J_m$ by \eqref{E:1}, thus we may apply Theorem \ref{MSS} with $r=2$ to obtain $J_{m+1}\subseteq J_m$ such that both $(T^{-1/2}_{m}...T^{-1/2}_1B_0^{-1/2}x_j)_{j\in J_{m+1}}$ and $(T^{-1/2}_{m}...T^{-1/2}_1B_0^{-1/2}x_j)_{j\in J_m\setminus J_{m+1}}$ have Bessel bounds $2^{-1}+ 2^{1/2}B_m^{-1/2}+B_m^{-1}$.
As $B_m\geq 200$, we have that this bound is smaller than 1. By
Theorem \ref{46} we have that
$(T^{-1/2}_{m}...T^{-1/2}_1B_0^{-1/2}x_j)_{j\in J_{m+1}}$ has lower
frame bound $2^{-1}- 2^{1/2}B_m^{-1/2}-B_m^{-1}>0$.  Thus the frame
operator $T_{m+1}$ of
$(T^{-1/2}_{m}...T^{-1/2}_1B_0^{-1/2}x_j)_{j\in J_{m+1}}$ has
$\|T_{m+1}\|\leq 2^{-1}+ 2^{1/2}B_m^{-1/2}+B_m^{-1}$ and
$\|T_{m+1}^{-1}\|\leq (2^{-1}- 2^{1/2}B_m^{-1/2}-B_m^{-1})^{-1}$
which satisfies inequality \eqref{E:2} and \eqref{E:3}.   We have that
\begin{align*}
\|T_{m+1}^{-1/2}...T^{-1/2}_1 B_0^{-1/2}\|^2 &\leq \|T_{m+1}^{-1/2}\|^2\|T^{-1/2}_{m}...T^{-1/2}_1B_0^{-1/2}\|^2\\
&\leq  \|T_{m+1}^{-1/2}\|^2 B_m^{-1} \quad\textrm{ by }\eqref{E:1}\\
&=  \|T_{m+1}^{-1}\| B_m^{-1} \quad\textrm{ as }T_{m+1}\textrm{ is a positive operator}\\
&\leq  (2^{-1}- 2^{1/2}B_m^{-1/2}-B_m^{-1})^{-1} B_m^{-1} \quad\textrm{ by }\eqref{E:3}\\
&=  (2^{-1}B_m- 2^{1/2}B_m^{1/2}-1)^{-1}=B_{m+1}^{-1}
\end{align*}
Thus, the inequality \eqref{E:1} is satisfied and our induction is
complete.

We have that $(T^{-1/2}_{n}...T^{-1/2}_1 B_0^{-1/2} x_j)_{j\in
J_{n}}$ is a Parseval frame and that $79\leq B_n<200$.  By Lemma \ref{L:operator} we have that
$(x_j)_{j\in J_n}$ is a frame with upper frame bound
$\|T^{1/2}_{n}...T^{1/2}_1 B_0^{1/2}\|^2$ and lower frame bound $\|T^{-1/2}_{n}...T^{-1/2}_1
B_0^{-1/2}\|^{-2}$.
By \eqref{E:2}, the upper frame bound of $(x_j)_{j\in J_n}$ is at most
$$ \|T^{1/2}_{n}...T^{1/2}_1 B_0^{1/2}\|^2\leq \|T_{n}\|...\|T_1\| B_0\leq B_0\prod_{0\leq m< n} (2^{-1}+ 2^{1/2}B_m^{-1/2}+B_m^{-1})=:B.$$

 By \eqref{E:3}, the lower frame bound of $(x_j)_{j\in J_n}$ is at least 
$$ \|T_{n}^{-1}\|^{-1}...\|T^{-1}_1\|^{-1} B_0\geq B_0\prod_{0\leq m< n} (2^{-1}- 2^{1/2}B_m^{-1/2}-B_m^{-1})= B_0\prod_{0\leq m< n} B_{m+1}B_m^{-1}=B_n=:A.$$

We now have an upper frame bound $B$ and lower frame bound $A$ for $(x_j)_{j\in J_n}$.
  If there exists a constant $C$ such that the ratio of the frame bounds $B/A$ is uniformly bounded by $C$, then we would have a lower frame bound of $A=B_n\geq 79$ and an upper frame bound of $B\leq AC\leq 200C$.  This would prove that every tight frame of vectors in the unit ball of a finite dimensional Hilbert with frame bound greater than $79$ can be partitioned into frames each of which has upper frame bound $200C$ and lower frame bound $79$.
Thus, all we need to prove is that $B/A$ is uniformly bounded.  We have that 
\begin{align*}
\ln (B/A)&= \ln(\prod \frac{ 2^{-1}+ 2^{1/2}B_m^{-1/2}+B_m^{-1}}{2^{-1}- 2^{1/2}B_m^{-1/2}-B_m^{-1}}) \\
&=\ln( \prod \frac{ 1+ 2^{3/2}B_m^{-1/2}+2B_m^{-1}}{1 - 2^{3/2}B_m^{-1/2}-2B_m^{-1}})\\
&\leq \ln(\prod 1+ 10 B_m^{-1/2} )\quad \textrm{ as }B_m\geq 79 \\
&=\sum \ln(1+ 10 B_m^{-1/2} )\\
&\leq\sum 10 B_m^{-1/2} \\
&\leq\sum_{m=0}^\infty 10 (79/200)^{m/2} 79 ^{-1/2}<\infty.
\end{align*}

Thus, we have proven that every  tight frame of vectors in the unit ball of a finite dimensional Hilbert space $H$ with frame bound greater than $1$ can be partitioned into a collection of frames of $H$ each with lower frame bound $A$ and upper frame bound $B$.   As part of the proof, we implicitly showed that for all $\vp>0$ there exists $a_\vp,b_\vp,D_\vp>0$ such that any tight frame of vectors in the unit ball of a finite dimensional Hilbert space with frame bound greater than $D_\vp$ may be partitioned into frames with frame bounds $a_\vp<b_\vp \leq D_\vp$ such that $b_\vp/a_\vp<1+\vp$.   We now show how this can be used to prove the moreover claim.  For $1\leq n\leq D_1$ we let $A_n=A$ and $B_n=B$.
Let $N\in\N$ such that $D_1< N$.  Choose $\vp>0$ to be the smallest value such that $D_\vp\leq N$.  Set $B_{N}=1+N^{-1}$ and $A_N=B_N (1+\vp)^{-1}$.  Note that $\lim A_N=\lim B_N=1$.  We now need to show that every tight frame of vectors with norm at most $N^{-1}$ and frame bound greater than 1 contains a subset with frame bounds $A_N$ and $B_N$.

Let 
$(x_j)_{j\in J}$ be a tight frame of vectors with norm at most $N^{-1}$ and frame bound greater than 1.    Thus,  $(N x_j)_{j\in J}$ is a tight frame of vectors in the unit ball of $H$ with frame bound greater than $D_\vp$ and hence may be partitioned into frames with frame bounds $a_\vp$ and $b_\vp$.  This gives a partition of $(x_j)_{j\in J}$ into frames $((x_j)_{j\in J_n})_{n\leq M}$ each with bounds $a_\vp N^{-2}$ and $b_\vp N^{-2}$.  Let $(x_j)_{j\in I}$ be a frame formed by combining frames in $((x_j)_{j\in J_n})_{n\leq M}$ such that $(x_i)_{i\in I}$ has the smallest possible upper frame bound greater than 1.  If we remove some frame $(x_j)_{j\in J_n}$ from $(x_j)_{j\in I}$ then the resulting frame $(x_j)_{j\in I\setminus J_n}$ has an upper frame bound of 1.  Thus, $(x_j)_{j\in I}$ has upper frame bound $B_N=1+N^{-1}$ as $(x_j)_{j\in J_n}$ has upper frame bound $b_\vp N^{-2}<N^{-1}$.  As $(x_j)_{j\in I}$ is a union of frames whose ratio of their frame bounds is at most $1+\vp$, we have that the ratio of the frame bounds of $(x_j)_{j\in I}$ is at most $1+\vp$.  Thus,  $(x_j)_{j\in I}$ has lower frame bound $A_N=B_N (1+\vp)^{-1}$.

\end{proof}

Theorem \ref{T:reduction} is stated only for tight frames with frame bound greater than 1.  The following corollary applies to partitioning any frame with lower frame bound greater than 1.

\begin{cor}\label{C:frame_part}
Let $(f_j)_{j\in J}$ be a frame of a Hilbert space $H$ with upper frame bound $B_0$ and lower frame bound $A_0\geq 1$ such that $\|f_j\|\leq1$ for all $j\in J$ then $(f_j)_{j\in J}$
can be partitioned into a collection of frames of $H$ each with lower frame bound $A$ and upper frame bound $B B_0 A_0^{-1}$.  Where $A$ and $B$ are the constants given in Theorem \ref{T:reduction}.
\end{cor}
\begin{proof}
Let $T$ be the frame operator of $(f_j)_{j\in J}$.  Then $\|T\|\leq B_0$ and $\|T^{-1}\|\leq A_0^{-1}$.  We have that $(A_0^{1/2}T^{-1/2}f_j)_{j\in J}$ is a tight frame with frame bound $A_0\geq1$.  For all $j\in J$ we have that $\|A_0^{1/2}T^{-1/2}f_j\|\leq A_0^{1/2}\|T^{-1}\|^{1/2}\|f_j\|\leq1$. By Theorem \ref{T:reduction} there is a   partition $(J_n)_{1\leq n\leq M}$ of $J$ such that $(A_0^{1/2}T^{-1/2}f_j)_{j\in J_n}$ has upper frame bound $B$ and lower frame bound $A$ for each $1\leq n\leq M$.
Let $x\in H$ and $1\leq n\leq M$.  Then,
$$\sum_{j\in J_n} |\langle f_j,x\rangle|^2=\sum_{j\in J_n} |\langle A_0^{1/2} T^{-1/2} f_j, A_0^{-1/2} T^{1/2}x\rangle|^2\leq B \|A_0^{-1/2}T^{1/2} x\|^2\leq B A_0^{-1} B_0\|x\|^2
$$
Thus, $(f_j)_{j\in J_n}$ has upper frame bound $B B_0 A_0^{-1}$.  We now check the lower frame bound.
$$\sum_{j\in J_n} |\langle f_j,x\rangle|^2=\sum_{j\in J_n} |\langle A_0^{1/2} T^{-1/2} f_j, A_0^{-1/2} T^{1/2}x\rangle|^2\geq A \|A_0^{-1/2}T^{1/2} x\|^2\geq A A_0^{-1} A_0\|x\|^2=A\|x\|^2
$$
Thus, $(f_j)_{j\in J_n}$ has lower frame bound $A$.

\end{proof}

We now restate and prove Corollary \ref{C:InfPart} from the Introduction, which proves that Theorem \ref{T:reduction} holds for infinite frames as well.
\begin{cor}
Let $A,B>0$ be the constants given in Theorem \ref{T:reduction1}. Then every tight frame of vectors in the unit ball of a  separable Hilbert space $H$ with frame bound greater than $1$ can be partitioned into a finite collection of frames of $H$ each with lower frame bound $A$ and upper frame bound $B$.
\end{cor}

\begin{proof}
Let $(f_j)_{j=1}^\infty$ be a tight frame of vectors in the unit ball of $H$ with frame bound $K>1$.   For each $n\in\N$ let $H_n=\textrm{span}_{1\leq j\leq n}f_j$ and let 
 $(g_{j,n})_{j\in I_n}$ be a finite collection of vectors in the ball of $H_n$ so that $(f_j)_{j=1}^n\cup (g_{j,n})_{j\in I_n}$ is a $K$-tight frame for $H_n$.
By Theorem \ref{T:reduction1} we have that $(f_j)_{j=1}^n\cup (g_{j,n})_{j\in I_n}$ may be partitioned into a collection of frames  $((f_j)_{j\in J_{i,n}}\cup (g_{j,n})_{j\in I_{i,n}})_{i=1}^{M_n}$ of $H_n$ each with lower frame bound $A$ and upper frame bound $B$.   We first obtain an upper bound on $M_n$.  For each $1\leq i\leq M_n$ we have that $(f_j)_{j\in J_{i,n}}\cup (g_{j,n})_{j\in I_{i,n}}$ has lower frame bound $A$ and that the entire collection of vectors $(f_j)_{j=1}^n\cup (g_{j,n})_{j\in I_n}$ has frame bound $K$.  Thus, we have that $AM_n\leq K$.  We let $M=\lfloor K/A\rfloor$.  Thus for each $n\in\N$ we may consider the partitioning to be of the form $((f_j)_{j\in J_{i,n}}\cup (g_{j,n})_{j\in I_{i,n}})_{i=1}^{M}$ where we allow for sets to be empty.

For a given $j\in\N$, we can have that the index $1\leq i\leq M$ so that $j\in J_{i,n}$ can change depending on $n$.  However, this can be stabilized by passing to a subsequence, which is what Pete Casazza refers to as the pinball principle.
We choose a subsequence $(k_n)_{n\in\N}$ of $\N$ so that for all $j\in\N$ there exists $1\leq m_j\leq M$ so that $j\in J_{m_j,k_n}$ for all $1\leq j\leq n$.  For each $1\leq i\leq M$, we let $J_i=\liminf_{n\rightarrow\infty} J_{i,k_n}$.  This gives that $(J_i)_{1\leq i\leq M}$ is a partition of $\N$.  We have that $(f_j)_{j\in J_i}$ has Bessel bound $B$ as $(f_j)_{j\in J_{i,n}}$ has Bessel bound $B$ for all $n\in\N$.
We now prove that if $J_i\not= \emptyset$ then $(f_j)_{j\in J_i}$ is a frame of $H$ with lower frame bound $A$.  
Let $x\in H$. We have that $\lim_{n\rightarrow\infty}\sum_{1\leq j\leq k_n} |\langle f_j,x \rangle|^2 =K\|x\|^2$.  Hence, $\lim_{n\rightarrow\infty}\sum_{j\in I_{k_n}} |\langle g_{j,k_n},x \rangle|^2=0$ as $(f_j)_{j=1}^n\cup (g_{j,n})_{j\in I_n}$ has frame bound $K$.  Thus for all $1\leq i\leq M$ we have that 
\begin{align*}
\sum_{j\in J_i}|\langle f_j,x\rangle|^2&=\lim_{n\rightarrow\infty} \sum_{j\in J_{i,k_n}}|\langle f_j,x\rangle|^2\\
&=\lim_{n\rightarrow\infty} \sum_{j\in J_{i,k_n}}|\langle f_j,x\rangle|^2+ \sum_{j\in I_{i,k_n}}|\langle g_{j,k_n},x\rangle|^2\\
&\geq A\|x\|^2
\end{align*}
Thus $(f_j)_{j\in J_i}$ has lower frame bound $A$.
\end{proof}

We note here that it is not possible to improve Theorem \ref{T:reduction} to show that every FUNTF with sufficiently many vectors contains a good \emph{basis}. A {\em FUNTF } (or  finite unit norm tight frame) is 
a finite collection of unit vectors which form a tight frame.  If $k\geq n$ are natural numbers then there always exists a FUNTF of $k$ vectors for an $n$-dimensional Hilbert space \cite{BF}, and FUNTFs are particularly  useful in application due to their resilience to error \cite{GKK}.

\begin{thm} For every $\epsilon > 0$ and every $B > 1$, there exists an $M > B$ and a FUNTF $(x_i)_{1\leq i \leq M}$ such that whenever $I\subset [1,M]$ is such that $(x_i)_{i\in I}$ is a basis, then the lower Riesz constant of $(x_i)_{i\in I}$ is less than $\epsilon$.
\end{thm}

\begin{proof} We modify slightly the construction of Casazza, Fickus, Mixon and Tremain from Proposition 3.1 in \cite{CFMT}. Let $H_n$  be the $2^n \times 2^n$ Hadamard matrix obtained via tensor products of
\[
\begin{pmatrix}
1&1\\1&-1
\end{pmatrix}.
\]
Let $F_n^1$ be the matrix obtained by multiplying the first $2^{n-1}-1$ columns of $H_n$ by $\sqrt{1/2^{n-1}}$, and the remaining $2^{n-1}+1$ columns by $\sqrt \frac 1{2^{n-1}(2^{n-1}+1)}$. Let $F_n^2$ be the matrix obtained by multiplying the first $2^{n-1} - 1$ columns of $H_n$ by 0 and the remaining $2^{n-1}+1$ columns by $\sqrt \frac 1{2^{n-1}+1}$.
Let $F_n$ be the $2^{n+1}\times 2^n$ matrix obtained by ``stacking" the $F_n^1$ on top of $F_n^2$.  Note that the $2^n$ columns of $F_n$ are orthogonal, the rows have norm one, and the columns have norm-squared 2. We denote the $j$th row of $H_n$ by $h_j$, the $j$th row of $F_n$ by $x_j$, and we note that $\{g_j = 2^{-n/2} h_j: 1\le j \le 2n\}$ is an orthonormal basis.

Let $I\subset [1,2^{n+1}]$ be of size $2^n$. We show that $(x_i)_{i \in I}$ cannot have good lower Riesz bound.  Case 1: If $|I \cap [1,2^n] | > 2^{n-1} - 1$, then $(x_i)_{i\in I}$ has lower Riesz bound less than $\frac{2}{2^{n-1} +1}$. Indeed, let $J = I\cap[ [1,2^n]$ and let $P$ denote the orthogonal projection onto the first $2^{n-1} -1$ coordinates. Choose scalars $(c_j)$ such that 
\[
\sum_{j\in J} |c_j|^2 = 1
\]
and
\[
P\bigl(\sum_{j\in J} c_j x_j \bigr) = 0.
\]
We then have,
\begin{align*}
\| \sum_{j\in J} c_j x_j\|^2 &= \|\sum_{j\in J} (Id - P)  c_j x_j \|^2 \\
  &=  \frac{2}{2^{n-1} + 1} \|(Id - P) \sum _{j\in J} c_j g_j\|^2 \\
  &\le  \frac{2}{2^{n-1} + 1} \| \sum _{j\in J} c_j g_j\|^2 \\
  &= \frac{2}{2^{n-1} + 1}.
\end{align*}
Thus the lower Riesz bound of $(x_j)_{j\in J}$ (and hence of $(x_j)_{j\in I}$ ) is no more than $\frac{2}{2^{n-1} + 1}$. 

Case 2: If $|I \cap [1,2^n] | < 2^{n-1} - 1$, then $(x_i)_{i\in I}$ cannot be a basis, since the projection onto the first $2^{n-1} - 1$ coordinates will not have $2^{n-1} - 1$ non-zero vectors, so it will not span.

It remains to examine what happens in Case 3: $|I\cap [1,2^n]| = 2^{n - 1} - 1$. In this case, $|I\cap [2^{n}+1,2^{n+1}]| =2^{n - 1} + 1$. In particular, there exist $x_j$ and $x_k$ such that the last $2^{n-1}$ coordinates are constant multiples of one another. So, $|\langle x_j,x_k\rangle| = \frac {2^{n-1}-1}{2^{n-1}+1}$. In particular, the Riesz constant of just these two (norm-one) vectors goes to 0 as $n\to \infty$. Therefore, we can choose $n$ such that the lower Riesz constant of any basis is less than $\epsilon$.

Finally, to finish the proof, we simply copy the FUNTF constructed above as many times as necessary to create a large enough FUNTF. Any basis contained in the copy will also be a basis in the original, so we can force the lower Riesz bound to be less than $\epsilon$.
\end{proof}

\section{Lemmas}\label{S:lem}

In this section we collect some lemmas on frames which will be necessary for solving the discretization problem in Section \ref{S:disc}.  The lemmas on continuous frames that we need will be presented in Section \ref{S:disc}.

{
\begin{lem}\label{L:BesselPart}
Let $(f_{j})_{j\in J}$ be a C-Bessel sequence in an $N$-dimensional Hilbert space $H$.  Let $M\in\N$ and $J_1,..., J_M$ be a partition of $J$.  For each $1\leq K< M$, there exists $1\leq n_1<...<n_K\leq M$ so that $(f_{j})_{j\in J_{n_k}}$ is $C N /(M+1-K)$-Bessel for each $1\leq k\leq K$.
\end{lem}

\begin{proof}
For a set $I\subseteq J$ we let $T_I$ be the frame operator of $(f_{j})_{j\in I}$.  For any choice of an orthonormal basis $(e_i)_{i=1}^N$ we have that
\begin{equation}\label{E:eigen}
trace(T_I)=\sum_{i=1}^N \sum_{j\in I} |\langle e_i, f_j\rangle|^2 = \sum_{j\in I} \|f_j\|^2.
\end{equation}
As $T_I$ is a positive self-adjoint operator, we may choose $(e_i)_{i=1}^N$ to be an orthonormal basis of eigenvectors of $T_I$, and hence \eqref{E:eigen} gives that  $\sum_{j\in I} \|f_j\|^2$ is equal to the sum of the eigenvalues of $T_I$.  For $J=I$, we have that the sum of the eigenvalues of $T_J$ is  $\sum_{i=1}^N \sum_{j\in J} |\langle e_i, f_j\rangle|^2\leq CN$ as $(f_j)_{j\in J}$ is $C$-Bessel.  Thus we may choose $K$ different $n\in \N$ so that the sum of the eigenvalues of $T_{J_n}$ is at most $CN/(M+1-K)$.  In particular, each of the eigenvalues of $T_{J_n}$ is at most $CN/(M+1-K)$ and $(f_j)_{j\in J_n}$ is $CN/(M+1-K)$ Bessel.
\end{proof}

\begin{cor}\label{C:BesselPart}
Let $P\in \N$. For $1\leq p\leq P$, let  $(f_{p,j})_{j\in J^p}$ be a $C_p$-Bessel sequence in an $N_p$-dimensional Hilbert space.  Let $M\in\N$ and $J_1^p,..., J_M^p$ be a partition of $J^p$. Then there exists $1\leq n\leq M$ such that the sequence $(f_{p,j})_{j\in J_{n}^p}$ is $4^{p} C_p N_p /M$-Bessel for each $1\leq p\leq P$.
\end{cor}
\begin{proof}
Without loss of generality we may assume that $2^{P+1}<M$.  Indeed, for $1\leq p\leq P$, if $M\leq 2^{p+1}$ then $(f_{p,j})_{j\in J_{n}^p}$ is automatically $4^{p} C_p N_p /M$-Bessel for all $n$.

We let $M_1=M$  if $M$ is even, and we let $M_1=M-1$ if $M$ is odd.  By Lemma \ref{L:BesselPart} we may  choose  $\{n_1,...,n_{M_1/2}\}$ such that $(f_{1,j})_{j\in J_{n_i}^1}$ is $2CN/M_1$-Bessel for all $1\leq i\leq M_1/2$.  Continuing, we let $M_2=M_1/2$ if $M_1/2$ is even, and we let $M_2=M_1/2-1$ if $M_1/2$ is odd.  We can then choose half of the set $\{n_1,...,n_{M_1/2}\}$ (without loss of generality, it is the first half) such that $(f_{2,j})_{j\in J_{n_i}^2}$ is $2^2CN/M_2$-Bessel for all $1\leq i\leq M_2/2$.  We then let 
$M_3=M_2/2$ if $M_2/2$ is even, and we let $M_3=M_2/2-1$ if $M_2/2$ is odd.  Without loss of generality, we have that $(f_{3,j})_{j\in J_{n_i}^3}$ is $2^3CN/M_3$-Bessel for all $1\leq i\leq M_3/2$.  Continuing in this manner we get even natural numbers $M_1,...,M_P$ with $2^{p-1} M_p\leq M\leq 2^{p}M_p$ for each $1\leq p\leq P$ such that 
$(f_{p,j})_{j\in J_{n_i}^p}$ is $2^pCN/M_p$-Bessel for all $1\leq i\leq M_p/2$, and hence 
$(f_{p,j})_{j\in J_{n_i}^p}$ is $4^pCN/M$-Bessel for all $1\leq i\leq M_p/2$.  Thus, we may take the sets $J_{n_1}^p$ as $2^{P+1}<M$ and $2^{P-1} M_P\leq M\leq 2^{P}M_P$ implies that $1\leq M_P$.
\end{proof}
}

We give a quick example showing that the estimate in Lemma \ref{L:BesselPart} is sharp for $K=1$.  Let $C>0$ and let $N,M\in \Z$ such that $N$ divides $M$.  Let $(e_i)_{i=1}^N$ be an orthonormal basis of an $N$-dimensional Hilbert space $H$.  Consider the $C$-tight frame which consists of $M/N$ copies of $\sqrt{\frac{CN}{M}}e_i$ for each $1\leq i\leq N$.  As there are $N$ choices for $1\leq i\leq N$, we have that the frame has $M$ vectors.  Thus if we partition the frame into singletons we have that each singleton is $\frac{CN}{M}$-Bessel.

The following lemma is obvious for Hilbert spaces, but we state it separately because it will be important in obtaining the lower frame bound in Theorem \ref{T:samp}.  

\begin{lem}\label{L:block}
Let $\vp>0$ and $N\in\N$ with $N\geq \vp^{-2} $.  Let $(H_j)_{j=1}^N$ be a sequence of finite dimensional mutually orthogonal spaces of a Hilbert space $H$.   Then for every $x\in H$ there exists $1\leq n\leq N$ such that $\|P_{H_n} x\|\leq \vp\|x\|$ where $P_{H_n}$ is orthogonal projection onto $H_n$.
Furthermore, if $N$ is even and  $N\geq 2\vp^{-2}$ then for every $x\in H$ there exists $1\leq n<N$ such that $\|P_{H_n\oplus H_{n+1}} x\|\leq \vp\|x\|$

\end{lem}

\begin{proof}
For the sake of contradiction we assume that there exists $x\in H$ with $\|P_{H_j} x\|>\vp\|x\|$ for all $1\leq j\leq N$.  This gives the following contradiction,
$$\|x\|\geq (\sum_{j=1}^N \|P_{H_j} x\|^2)^{1/2}>  (\sum_{j=1}^N \vp^2\|x\|^2)^{1/2}
=N^{1/2}\vp\|x\|\geq \|x\|$$.

For the furthermore case, we let $H^0_n=H_{2n-1}\oplus H_{2n}$ for all $n\in \N$ and then apply the previous case to $N/2\in \N$ and $(H^0_{j})_{j=1}^{N/2}$.
\end{proof}

\begin{lem}\label{L:project}
Let $H_0$ and $H_1$ be Hilbert spaces and let $(f_j)_{j\in J}\subset H_0\oplus H_1$.  Suppose that $(P_{H_1} f_j)_{j\in J}\subseteq H_1$ has upper frame bound $K$ and lower frame bound $k$ and that $(P_{H_0} f_j)_{j\in J}$ has Bessel bound $c$.  Then for all $x\in H_0\oplus H_1$,
$$\sum_{j\in J} |\langle f_j, x\rangle|^2 \leq K\|P_{H_1}x\|^2+c\|P_{H_0}x\|^2+2 K^{1/2}c^{1/2}\|P_{H_1}x\|\|P_{H_0}x\|$$
and
\begin{align*}
\sum_{j\in J} |\langle f_j, x\rangle|^2 &\geq \left(\sum_{j\in J} |\langle f_j, P_{H_1}x\rangle|^2\right) - 2 K^{1/2}c^{1/2}\|P_{H_1}x\|\|P_{H_0}x\| \\
 &\ge k\|P_{H_1}x\|^2-2 K^{1/2}c^{1/2}\|P_{H_1}x\|\|P_{H_0}x\|.
 \end{align*}
\end{lem}

\begin{proof}
We first calculate the upper bound.
\begin{align*}\sum_{j\in J}& |\langle f_j, x\rangle|^2 = \sum_{j\in J} |\langle f_j, P_{H_1}x+P_{H_0}\rangle|^2 \\
 &\leq \sum_{j\in J} |\langle f_j, P_{H_1}x\rangle|^2+|\langle f_j,P_{H_0}x\rangle|^2+2|\langle f_j, P_{H_0}x\rangle\langle f_j, P_{H_1}x\rangle| \\
 &\leq \left(\sum_{j\in J} |\langle f_j, P_{H_1}x\rangle|^2+|\langle f_j,P_{H_0}x\rangle|^2\right)+2\left(\sum_{j\in J}|\langle f_j, P_{H_0}x\rangle|^2\right)^{1/2}\left(\sum_{j\in J}|\langle f_j, P_{H_1}x\rangle|^2\right)^{1/2} \\
&= \left(\sum_{j\in J} |\langle P_{H_1}f_j, P_{H_1}x\rangle|^2+|\langle P_{H_0}f_j,P_{H_0}x\rangle|^2\right)+2\left(\sum_{j\in J}|\langle P_{H_0} f_j, P_{H_0}x\rangle|^2\right)^{1/2}\left(\sum_{j\in J}|\langle P_{H_1}f_j, P_{H_1}x\rangle|^2\right)^{1/2} \\
&\leq  K\|P_{H_1}x\|^2+c\|P_{H_0}x\|^2+2 K^{1/2}c^{1/2}\|P_{H_1}x\|\|P_{H_0}x\|
\end{align*}
We now calculate the lower bound.
\begin{align*}\sum_{j\in J}& |\langle f_j, x\rangle|^2 = \sum_{j\in J} |\langle f_j, P_{H_1}x+P_{H_0}\rangle|^2 \\
 &\geq \sum_{j\in J} |\langle f_j, P_{H_1}x\rangle|^2+|\langle f_j,P_{H_0}x\rangle|^2-2|\langle f_j, P_{H_0}x\rangle\langle f_j, P_{H_1}x\rangle| \\
 &\geq \left(\sum_{j\in J} |\langle f_j, P_{H_1}x\rangle|^2\right)-2\left(\sum_{j\in J}|\langle f_j, P_{H_0}x\rangle|^2\right)^{1/2}\left(\sum_{j\in J}|\langle f_j, P_{H_1}x\rangle|^2\right)^{1/2} \\
&=  \sum_{j\in J} |\langle f_j, P_{H_1}x\rangle|^2 -2\left(\sum_{j\in J}|\langle P_{H_0} f_j, P_{H_0}x\rangle|^2\right)^{1/2}\left(\sum_{j\in J}|\langle P_{H_1}f_j, P_{H_1}x\rangle|^2\right)^{1/2} \\
&\ge \sum_{j\in J} |\langle f_j, P_{H_1}x\rangle|^2  - 2 K^{1/2}c^{1/2}\|P_{H_1}x\|\|P_{H_0}x\|\\
&\geq  k\|P_{H_1}x\|^2-2 K^{1/2}c^{1/2}\|P_{H_1}x\|\|P_{H_0}x\|
\end{align*}
\end{proof}

\begin{lem}\label{L:orthogonal} Let $\vp>0$ and $\vp_1 = 3\vp + 2(2\vp)^{1/2}(1 + \vp)^{1/2}$.  Suppose $(f_j)_{j\in J}\subset H_0\oplus H_1$ is a $(1+\vp)$-Bessel sequence and $(P_{H_1}f_j)_{j\in J}\subset H_1$ has lower frame bound $1-\vp$.  Then there exists  $(g_i)_{i\in I}\subseteq H_0$ such that $(g_i)_{i\in I}\cup (f_j)_{j\in J}$ is $(1+\vp_1)$-Bessel and has lower frame bound $(1-\vp_1)$ on $H_0\oplus H_1$.
\end{lem}
\begin{proof}
Let $(h_i)_{i\in I}\subset H_0\oplus H_1$ such that $(h_i)\cup(f_j)$ is a $(1+\vp)$-tight frame.  As, $(P_{H_1} f_j)\subset H_1$ has lower frame bound $(1-\vp)$ and $(P_{H_1} h_i)\cup(P_{H_1} f_j)\subset H_1$ is $(1+\vp)$-tight, we have that $(P_{H_1} h_i)$ is $2\vp$-Bessel.
We will prove that $(P_{H_0}h_i)_{i\in I}\cup (f_j)_{j\in J}$ has upper frame bound $(1+\vp_1)$.  Let $x\in H$.
\begin{align*}
\sum |\langle P_{H_0}h_i,x\rangle|^2 +& \sum |\langle f_j,x\rangle|^2=
  (1+\vp)\|x\|^2 - (\sum |\langle (P_{H_0}+P_{H_1})h_i,x\rangle|^2 - \sum |\langle P_{H_0} h_i,x\rangle|^2)\\
&\leq
  (1+\vp)\|x\|^2 + (\sum |\langle P_{H_1}h_i,x\rangle|^2 + 2\sum |\langle P_{H_1} h_i,x\rangle \langle P_{H_0} h_i,x\rangle|)\\
&\leq
  (1+\vp)\|x\|^2 + (\sum |\langle P_{H_1}h_i,x\rangle|^2 + 2(\sum |\langle P_{H_1} h_i,x\rangle|^2)^{1/2}(\sum |\langle P_{H_0} h_i,x\rangle|^2)^{1/2})\\
&\leq
  (1+\vp)\|x\|^2 + (2\vp + 2(2\vp)^{1/2}(1+\vp)^{1/2})\|x\|^2\leq (1+\vp_1)\|x\|^2
\end{align*}
We now prove that $(P_{H_0}h_i)_{i\in I}\cup (f_j)_{j\in J}$ has lower frame bound $(1-\vp_1)$, which follows the same argument as above. Let $x \in H_0 \oplus H_1$.
\begin{align*}
\sum |\langle P_{H_0}h_i,x\rangle|^2 +& \sum |\langle f_j,x\rangle|^2=
  (1+\vp)\|x\|^2 - (\sum |\langle (P_{H_0}+P_{H_1})h_i,x\rangle|^2 - \sum |\langle P_{H_0} h_i,x\rangle|^2)\\
&\geq
  (1+\vp)\|x\|^2 - (\sum |\langle P_{H_1}h_i,x\rangle|^2 + 2\sum |\langle P_{H_1} h_i,x\rangle \langle P_{H_0} h_i,x\rangle|)\\
&\geq
  (1+\vp)\|x\|^2 - (\sum |\langle P_{H_1}h_i,x\rangle|^2 + 2(\sum |\langle P_{H_1} h_i,x\rangle|^2)^{1/2}(\sum |\langle P_{H_0} h_i,x\rangle|^2)^{1/2})\\
&\geq
  (1+\vp)\|x\|^2 - (2\vp + 2(2\vp)^{1/2}(1+\vp)^{1/2})\|x\|^2\geq (1-\vp_1)\|x\|^2
\end{align*}

\end{proof}

\section{Continuous frames and the discretization problem}\label{S:disc}

Recall that a measurable function  $\Psi:X\rightarrow H$ from a measure space with a $\sigma$-finite measure $\mu$ to a separable Hilbert space $H$ is called a  {\em continuous frame} with respect to $\mu$ if there exist constants $A,B>0$ so that
\begin{equation}\label{E:def}
A\|x\|^2\leq \int |\langle x, \Psi(t)\rangle|^2 d\mu(t)\leq B \|x\|^2\quad\quad\forall x\in H.
\end{equation}
If $A=B$ then the continuous frame is called {\em tight} and if $A=B=1$ then the continuous frame is called {\em Parseval} or a {\em coherent state}.
We say that $\Psi:X\rightarrow H$ is a {\em continuous Bessel map} if it does not necessarily have a positive lower frame bound $A$, but does have  a finite upper frame bound $B$, which is also called a {\em Bessel bound}.
As with frames and Bessel sequences in Hilbert spaces, a continuous frame or continuous Bessel map for a Hilbert space induces a bounded  positive operator $T:H\rightarrow H$ called the {\em frame operator} which is  defined by
\begin{equation}\label{E:frameOp}
T(x)=\int \langle x, \Psi(t)\rangle \Psi(t) \,d\mu(t)\quad\quad\forall  x\in H.
\end{equation}
We are integrating vectors in a Hilbert space and  \eqref{E:frameOp} is defined weakly in terms of the Pettis integral.  
That is,  for all $x\in H$, $T(x)$ is defined to be the unique vector such that 
\begin{equation}\label{E:Pettis}
\langle T(x), y\rangle=\int \langle x, \Psi(t)\rangle \langle \Psi(t), y\rangle \,d\mu(t) \quad\quad\forall  y\in H.
\end{equation}

 It is sometimes more convenient to work with the inequalities in \eqref{E:def} and sometimes it will be useful to work with the frame operator $T$.  As with discrete frames, $\|T\|=B$  where $B$ is the optimal upper frame bound, and $\Psi$ is a continuous frame if and only if $T$ is invertible and in which case $\|T^{-1}\|=A^{-1}$ where $A$ is the optimal lower frame bound.
Given a frame for a Hilbert space, it may be converted to a Parseval frame by applying the inverse of the square root of the frame operator.  The following lemma shows that this same technique works for continuous frames.

\begin{lem}\label{L:parseval}
Let $\Psi:X\rightarrow H$ be a continuous frame with frame operator $T:H\rightarrow H$.  Then $T^{-1/2}\Psi:X\rightarrow H$ is a continuous Parseval frame.
\end{lem}
\begin{proof}
As $T$ is a positive self adjoint invertible linear operator, we have that the inverse of its square root $T^{-1/2}$ is well defined.  For $x\in X$ we have that
\begin{align*}
\|x\|^2=&\langle T (T^{-1/2}x),T^{-1/2}x\rangle \quad\textrm{ as $T$ is self-adjoint}\\
=&\int \langle T^{-1/2}x, \Psi(t)\rangle \langle\Psi(t),T^{-1/2}x\rangle \,d\mu(t)\quad\textrm{ by \eqref{E:Pettis}}\\
=&\int \langle x, T^{-1/2}\Psi(t)\rangle \langle T^{-1/2}\Psi(t),x\rangle \,d\mu(t)\quad\textrm{ as $T$ is self-adjoint}\\
=&\int |\langle x, T^{-1/2}\Psi(t)\rangle|^2  \,d\mu(t)
\end{align*}
Thus, $T^{-1/2}\Psi:X\rightarrow H$ is a continuous Parseval frame.
\end{proof}

Continuous frames were developed by Ali, Antoine, and Gazeau in \cite{AAG1}  as a generalization of coherent states and in their later textbook \cite{AAG2} they asked the following question which is now known as the Discretization Problem.

\begin{problem}[The Discretization Problem]
Let  $\Psi:X\rightarrow H$ be a continuous frame.  When does there exist a countable set $F\subseteq X$ such that $(\Psi(t))_{t\in F}\subseteq H$ is a frame of $H$?
\end{problem}
The Discretization Problem essentially asks if one can always obtain a frame for a Hilbert space from a continuous frame by sampling.  A solution for certain types of continuous frames was obtained by Fornasier and Rauhut using the theory of co-orbit spaces \cite{FR}.  

The following lemma allows us to approximate any continuous Bessel map with a continuous Bessel map having countable range.

\begin{lem}\label{L:discretize}
Let  $\Psi:X\rightarrow H$ be a continuous
Bessel map.  For all
$\vp>0$,  there exists a measurable partition $(X_j)_{j\in J}$ of $X$
and $(t_j)_{j\in J}\subset X$ such that $t_j\in X_j$ for all $j\in
J$ and $\|\Psi(t)-\Psi(t_j)\|<\vp$ for all $t\in X_j$ and
$$\int \left\| \Psi(t)-\sum_{j\in J}\Psi(t_j) 1_{X_j}(t)\right\| d\mu(t)<\vp.
$$
\end{lem}

\begin{proof}
We first claim that we may assume  that $\mu$ is
non-atomic.  Indeed, if $(E_i)_{i\in I}$ is a
collection of disjoint atoms such that $X\setminus \cup_{i\in I}
E_i$ is non-atomic, then $\Psi$ is constant almost everywhere on
each $E_i$.  Thus for each $i\in I$, there exists $s_i\in E_i$ such
that $\Psi 1_{E_i}=\Psi(s_i) 1_{E_i}$ almost everywhere. Hence we
would only need to prove the lemma for the non-atomic measure space
$X\setminus\cup_{i\in I}E_i$ and the continuous Bessel map
$\Psi|_{X\setminus \cup_{i\in I} E_i}=\Psi-\sum_{i\in I}\Psi(t)
1_{E_i}$.  Thus, we assume without loss of generality that $X$ is
non-atomic.

Let $\vp>0$. As $X$ is non-atomic and $\sigma$-finite, $X$ may be partitioned
into a sequence of pairwise disjoint measurable subsets
$(Y_j)_{j\in \N}$ so that $\mu(Y_j)\leq1$ for all $j\in \N$. For all
$j\in \N$, let $(H^j_n)_{n=1}^\infty$ be a partition of
$\Psi(Y_j)\subseteq H$ such that $diam(H_n^j)<\vp2^{-j}$ for all
$n\in\N$.  For each $j,n\in \N$ choose $t^j_{n}\in
\Psi^{-1}(H^j_n)\cap Y_j$. Note that $(\Psi^{-1}(H^j_n)\cap Y_j)_{n\in\N}$ is a partition of $Y_j$ for all $j\in\N$ and hence $(\Psi^{-1}(H^j_n)\cap Y_j)_{j,n\in\N}$ is a partition of $X$.  We have that $\|\Psi(t)-\Psi(t^j_n)\|<\vp2^{-j}$ for all $n,j\in\N$ and $t\in \Psi^{-1}(H_n^j)\cap Y_j$.  
We now estimate the following.
\begin{align*}
\int \left\| \Psi(t)-\sum_{j,n\in\N} \Psi(t_n^j)
1_{\Psi^{-1}(H^j_n)\cap Y_j}(t)\right\| d\mu(t)
&=\sum_{j,n\in\N}\int_{\Psi^{-1}(H^j_n)\cap Y_j} \left\| \Psi(t)-\Psi(t_n^j) \right\| d\mu(t)\\
&\leq\sum_{j,n\in\N}\mu(\Psi^{-1}(H^j_n)\cap Y_j) \vp
2^{-j}=\sum_{j\in\N}\mu(Y_j) \vp 2^{-j}\leq\vp
\end{align*}
Thus we may use  $(\Psi^{-1}(H^j_n)\cap Y_j)_{j,n\in\N}$ as our partition of $X$ and $(t^j_n)_{j,n\in\N}$ as our sample points.
\end{proof}

The following is a technical result that is needed in our proof of the discretization theorem, which may be of independent interest. Recall that a collection of vectors $(x_i)_{i \in I}$ in $H$ is said to be a scalable frame  if there exist constants $(c_i)_{i\in I}$ such that $(c_ix_i)_{i\in I}$ is a Parseval frame for $H$ \cite{KOPT}. The following result implies, in particular, that there are universal constants $A$ and $B$ such that if $(x_i)_{i\in I}$ is a scalable frame in the unit ball of $H$, then $(x_i)_{i\in I}$ can be sampled to form a frame with lower frame bound $A$ and upper frame bound $B$.  In order to help the reader stay organized, there are several claims in the proof of the theorem, whose proofs are separated out from the main text of the proof of the theorem. The proofs of the claims end in $\blacksquare$, while the end of the proof of the main theorem ends with $\square$, as usual.

\begin{thm}\label{T:samp} There is a function $g:(0,1/256)\to (0,1)$ such that $\lim_{\vp \to 0} g(\vp) = 0$ and if $(x_n)_{n=1}^\infty \subset \{x\in H: \|x\|\le 1\}$ is such that there exists scalars $(a_n)_{n=1}^\infty$ such that $(a_n x_n)_{n=1}^\infty$ is a frame for $H$ with bounds $1 - \vp$ and $1 + \vp$, then there exists $f:\N \to \N$ such that $(x_{f(n)})_{n=1}^\infty$ is a frame for $H$ with bounds $A(1 - g(\vp))$ and $2 B(1 + g(\vp))$, where $A$ and $B$ are the constants given in Theorem \ref{T:reduction}.
\end{thm}

\begin{proof} We will not explicitly define the functions $f$ and $g$, though a careful reading of the proof will give exact values for $g$.  The general strategy of the proof is to decompose $H$ into pairwise orthogonal subspaces and sample $(x_n)$ in order to obtain frames for the subspaces, while controlling the leakage into the other subspaces. 

Let $0 < \vp < 1/256$. We begin by noting that we may assume that $(a_n)_{n=1}^\infty$  are non-negative real numbers and by perturbing (and replacing $\vp$ by a slightly larger value, which we still denote as $\vp$) we may assume further that $(a_n)_{n=1}^\infty$ are non-negative rational. Next, we can decompose $H$ in the following way.

\begin{claim}\label{C:one} Let $K_{-1} = -1$ and $K_0 = 0$. There exists pairwise orthogonal subspaces $(H_n)_{n=1}^\infty$, an increasing sequence of natural numbers $(K_n)_{n=1}^\infty$ and a decreasing sequence of real numbers $(\vp_n)_{n=1}^\infty$ converging to zero such that the following hold for all $n\in \N$ and $1 \le m\le n$, where $\vp^\prime = 6\vp + 4\vp^{1/2} (1 + 2\vp)^{1/2}.$
\begin{align}
&(x_j)_{j \le K_{n-1}} \subset \oplus_{j\le n} H_j \label{Een} \\
&\textrm{ for all }  x\in \oplus_{m\le j \le n} H_j,
\,\,\,  (1 - 2\vp) \|x\|^2 \le \sum_{j\in (K_{m-2},K_n]} |\langle x, x_j \rangle|^2 a_j^2 \le (1 + \vp) \|x\|^2 \label{Twee} \\
&\textrm{ for all }  x\in \oplus_{1\le j \le n} H_j,
\,\,\,  \sum_{j\not \in (K_{m-2},K_n]} |\langle x, x_j \rangle|^2 a_j^2 \le \vp_n \|x\|^2 \label{Drie} \\
&\vp_n B(1 + \vp^\prime)(1 - \vp^\prime)^{-2} \dim(\oplus_{j \le n} H_j) < \vp^\prime 8^{-n} \label{Vier}
\end{align}

\end{claim}

\renewcommand{\qedsymbol}{$\blacksquare$}

\begin{proof}[Proof of claim \ref{C:one}] 

For the base case of $n=1$ we have that
\eqref{Een} will be automatically satisfied as we have set
$K_0=0$.  We let $H_1=0$, $K_1$=1 and $\vp_1 = \vp/2$, which trivially
satisfies \eqref{Twee}, \eqref{Drie} and \eqref{Vier}.

For the induction step we let $k\in\N$ and assume that
\eqref{Een}, \eqref{Twee}, \eqref{Drie}, and \eqref{Vier} are true
for $m\leq n=k$. We choose $\vp_{k+1}<\vp_k$ small enough so that \eqref{Vier} is satisfied. Let $H_{k+1}=\textrm{span}_{K_{k-1}<i \leq
K_k}P_{(\oplus_{j\leq k}H_{j})^\perp} x_i$.  Thus, $H_{k+1}$
is orthogonal to $\textrm{span}_{j\leq k} H_j$ and
$\{x_j\}_{j\leq K_{k}}\subseteq \textrm{span}_{j\leq k+1} H_j$
which satisfies \eqref{Een}. As $(a_j x_j)_{j\in\N}$ is Bessel
and $\oplus_{i\leq k+1}H_{i}$ is finite dimensional, we may choose $K_{k+1}>K_k$ so
that $\sum_{j>K_{k+1}} |\langle x, x_j)\rangle|^2 a_j^2\leq
\vp_{k+1}\|x\|^2$ for all $x\in \oplus_{i\leq k+1} H_{i}$. Let $m\leq k+1$. As $\{x_j\}_{j\leq
K_{m-2}}\subseteq \textrm{span}_{j\leq m-1} H_j$ we have that $\langle
x, x_j\rangle=0$ for all $x\in  \oplus_{m\leq i\leq k+1} H_{i}$ and $j \le K_{m-2}$;  hence,
\eqref{Drie} is true.  We have that \eqref{Twee} follows
from \eqref{Drie} as $({a_j}x_j)_{j\in\N}$ is a frame with
lower frame bound $1-\vp$ and upper frame bound $1+\vp$ and $\vp_n<\vp/2$.  Thus our  induction argument is complete.
\end{proof}

The spaces $(H_n)_{n=1}^\infty$ are the building blocks for constructing our frame via sampling, but we will need to group subspaces together in order to control the leakage between consecutive subspaces using Lemma \ref{L:block}. To this end, let $(\delta_n)_{n=1}^\infty$ be a decreasing sequence of real numbers and let $(M_n)_{n=1}^\infty$ and $(N_n)_{n=1}^\infty$ be sequences of odd numbers such that
\begin{align}
&M_1 < M_2 - 2 < M_2 < N_1 < N_1 + 1 < M_3 - 2 < M_3 < N_2 < N_2 + 1 < M_4 - 2 <  \cdots \label{Vyf}\\
&N_n - M_{n+1} - 4 > 2 \delta_n^{-2} \label{Ses}\\
&\sum_{n=1}^\infty \delta_n^2 < \vp^2 \label{Sewe}
\end{align}

We construct our frame via sampling as follows. Fix $r\in \N$. Since $\vp < 1/256$, we can choose $D = D_r$ to be an integer multiple of the least common multiple of the denominators of $(a_n^2)_{n\in (K_{M_r-2},K_{N_r}]}$ so that
\begin{equation}
D(1 - 6\vp - 4\vp^{1/2}(1 + 2\vp)^{1/2}) \ge 1 \label{Sewehalf}
\end{equation}
(this choice of $D$ will allow us to apply Corollary \ref{C:frame_part} in the sequel).
 Denote by $(\sqrt{1/D} f_n)_{n\in I}$ the sequence of vectors comprised of $D a_n^2$ copies of $\sqrt{1/D} x_j$ for each $j\in (K_{M_r-2}, N_r]$, where $I = I_r$ depends on $r$. We note here that for any $x\in H$, 
\begin{align*}
\sum_{n\in I} |\langle x, \sqrt{1/D} f_n\rangle|^2 &= \sum_{j = K_{M_r - 1}}^{K_{N_r}} \sum_{i = 1}^{Da_n^2} |\langle x, x_j\rangle|^2 \frac 1D\\
&=\sum_{j = K_{M_r - 1}}^{K_{N_r}} |\langle x, x_j \rangle|^2 a_n^2,
\end{align*}
so $(\sqrt{1/D} f_n)_{n\in I}$ has the same frame properties as $(a_n x_n)_{n\in (K_{M_r-2},K_{N_r}]}$. It will be necessary in the sequel to recover the correspondence between $f_j$ and $x_n$. So, we define 
\begin{equation}
b_r = b:I \to (K_{M_r-2}, K_{N_r}] \label{DefB}
\end{equation} 
in such a way that
\begin{enumerate}
\item $f_n = x_{b(n)}$ for all $n\in I$, and
\item $|b^{-1}(j)| = Da_j^2$ for all $K_{M_r - 1} \le j \le K_{N_r}$
\end{enumerate}
so that formally $(f_n)_{n\in I}$ is the same sequence of vectors as $(x_{b(n)})_{n \in I}$.

By \eqref{Een}, $(\sqrt{1/D} f_n)_{n\in I}$ is a $(1+\vp)$-Bessel sequence in $\oplus_{j \le {N_r } + 1} H_j$, and by \eqref{Twee}, 
\[
(P_{\oplus_{M_r\le j \le N_r} H_j} \sqrt{1/D} f_n)_{n\in I}
\]
is a frame for $\oplus_{M_r\le j \le N_r} H_j$ with lower frame bound $1 - 2\vp$. Therefore, we can apply Lemma \ref{L:orthogonal}  to obtain $(g_n)_{n\in J} \subset \oplus_{j < M_r} H_j \oplus H_{N_r + 1}$ such that $\|g_n\| \le 1/\sqrt{D}$ and $(\sqrt{1/D} f_n)_{n\in I} \cup (g_n)_{n\in J}$ is a frame for $\oplus_{j \le N_{r}+1} H_j$ with bounds $1 - \vp^\prime$ and $1 + \vp^\prime$, where 
\[
\vp^\prime = 6\vp + 4\vp^{1/2} (1 + 2\vp)^{1/2}.
\]

By our choice of $D$ \eqref{Sewehalf}, we can apply Corollary \ref{C:frame_part} to the vectors $(f_n)_{n\in I} \cup (\sqrt{D} g_n)_{n\in J}$ to obtain a partition $(I_m \cup J_m)_{m=1}^M$ of $I \cup J$ such that for each $1\le m \le M$, $(f_n)_{n\in I_m} \cup (\sqrt{D} g_n)_{n\in J_m}$ is a frame for $\oplus_{j\le N_r + 1} H_j$ with constants $A$ and $B(1 + \vp^\prime)(1 - \vp^\prime)^{-1}$. For a set $P \subset I$, we denote
\[
P^{>k} = \{j \in P: b(j) > K_k\},
\]
where $b$ is as defined in \eqref{DefB}. 

\begin{claim}\label{C:two} For each $r \ge 1$ there exists $m_0$ such that
\begin{align}
&\textrm{ for each } k\in [M_r-2, N_r], (P_{\oplus_{j \le k} H_j} f_n)_{n\in I_{m_0}^{>k}} \textrm{ has Bessel bound } \vp^\prime 2^{-r} \label{SewePlus},\\
&(P_{\oplus_{j\le M_r-2} H_j} f_n)_{n\in I_{m_0}} \textrm{ has Bessel bound } \vp^\prime 2^{-r} \label{Agt}, \\
&(P_{\oplus_{j\in [M_r, N_r]} H_j} f_n)_{n\in I_{m_0}} \textrm{ is a frame for } \oplus_{j\in [M_r, N_r]} H_j \textrm{ with bounds } A, B(1 + \vp^\prime)(1 - \vp^\prime)^{-1} \label{Nege}, \\
&(P_{\oplus_{j\in [M_r - 1, N_r]} H_j} f_n)_{n\in I_{m_0}} \textrm{ has Bessel bound } B(1 + \vp^\prime)(1 - \vp^\prime)^{-1}, \label{Tien} \\
&(f_n)_{n\in I_{m_0}} \textrm{ has Bessel bound } B(1 + \vp^\prime)(1 - \vp^\prime)^{-1} \textrm { on } \oplus_{1\le i \le N_r + 1} H_i. \label{TienPlus}
\end{align}
\end{claim}

\begin{proof}[Proof of Claim \ref{C:two}]
For each $k\in [M_r-2,N_r]$, we have that $(P_{\oplus_{1\leq i\le k}H_i}f_j)_{j\in I^{>k}}$ is
$D\vp_{k}$-Bessel by \eqref{Drie}. 
By Corollary \ref{C:BesselPart}
there exists $1\leq m_0\leq M$ (which depends on $r\in\N$) so that for every $k\in [M_r - 2, N_r]$, $(P_{\oplus_{i\le k}H_i}f_j)_{j\in I^{>k}_{m_0}}$ has Bessel bound $D\vp_k 4^k
dim(\oplus_{i\le k}H_i)/M$. 

As
each frame in the partition   $ ((f_j)_{j\in I_k}\cup(\sqrt{D} g_j)_{j\in
J_k})_{1\leq k\leq M}$ has upper frame bound $B(1+\vp^\prime)(1-\vp^\prime)^{-1}$
and there are $M$ of them, we have that $BM(1+\vp^\prime)(1-\vp^\prime)^{-1}\geq
(1-\vp^\prime)D$. Thus we have for each $k\in[M_r-2,N_r]$ that  
\begin{gather*}
(P_{\oplus_{ i\le k}H_i}f_j)_{j\in I^{>k}_{m_0}}\textrm{ has Bessel bound }\vp^\prime2^{-k}<\vp^\prime2^{-r}\textrm{, as }\\
D\vp_k 4^{k} dim(\oplus_{i\le k}H_i)/M<B(1 + \vp^\prime)(1-\vp^\prime)^{-2}\vp_k 4^{k} dim(\oplus_{i\le k}H_i)<\vp^\prime 2^{-k}\,\,\textrm{ by \eqref{Vier}},
\end{gather*}
which proves \eqref{SewePlus}.
When $k=M_r-2$ we have that $I_{m_0}=I_{m_0}^{>M_r-2}$ and hence, 
\begin{equation*}
(P_{\oplus_{ i\le M_r-2}H_i}f_j)_{j\in I_{m_0}}\textrm{ has Bessel bound }\vp^\prime2^{-r},
\end{equation*}
which proves \eqref{Agt}.

Equation \eqref{Nege} follows from the frame bounds of $(f_n)_{n\in I_m} \cup (\sqrt{D} g_n)_{n\in J_m}$ and the fact that $(g_n)_{n\in J_m}$ is orthogonal to $\oplus_{j\in [M_r, N_r]} H_j$. Equation \eqref{Tien} follows immediately from construction, and \eqref{TienPlus} follows from \eqref{Tien} and \eqref{Een}.
 \end{proof}

For each $r\in \N$, we define $I(r)$ to be the $I_m$ that is guaranteed to exist in Claim \ref{C:two}. To finish the proof, we show that $(f_n)_{n\in I(r), r\in \N}$ is a frame with bounds $2 B(1 + \vp_2)$ and $A - (A + 2)\vp_2 - 4B(1 + \vp_2)\vp_2^{1/2}$, where
\[
\vp_2 = (1 + \vp^\prime)(1 - \vp^\prime)^{-1} - 1 + \vp^\prime + 2(B(1 + \vp^\prime)(1 - \vp^\prime)^{-1}\vp^\prime)^{1/2}.
\]

\begin{proof}[Proof of Upper Bound] Let $x\in H$ with $\|x\| = 1$. Define a sequence of integers $(q_n)_{n=1}^\infty$ by $q_1 = 0$ and $q_n \in (N_{n-1} + 1 , M_{n+1}-2)$ for $n\ge 2$. Let
\begin{gather*}
I_-(n) = \{j\in I(n): b(j) \le K_{q_n - 1}\},\\
I_+(n) = \{j\in I(n): b(j) > K_{q_n - 1}\}.
\end{gather*}
First, note that if $j\in I_-(n)$, then $f_j \in \oplus_{1\le i \le q_n} H_i$ by \eqref{Een}. Therefore, defining $H_1^n$ to be $\oplus _{M_n-1\le i \le q_n} H_i$ we can apply \eqref{Agt}, \eqref{Tien} and Lemma \ref{L:project} to obtain that
\begin{align*}
I&:=\sum_{n=1}^\infty \sum_{j\in I_-(n)} |\langle f_j, x\rangle|^2 \\
&\le \sum_{n=2}^\infty \biggl(B(1 + \vp^\prime) (1 - \vp^\prime)^{-1} \| P_{H_1^n} x\|^2 + \vp^\prime 2^{-n} + 2\bigl(B(1 + \vp^\prime)(1 - \vp^\prime)^{-1} \vp^\prime 2^{-n}\bigr)^{1/2} \| P_{H_1^n} x\|\biggr)\\
&\le B(1 + \vp^\prime)(1 - \vp^\prime)^{-1} + \vp^\prime + 2\bigl(B(1 + \vp^\prime)(1 - \vp^\prime)^{-1} \vp^\prime \bigr)^{1/2} \sum_{n=2}^\infty 2^{-n} \sum_{n=2}^\infty \| P_{H_1^n}x \|^2\\
&\le  B(1 + \vp^\prime)(1 - \vp^\prime)^{-1} + \vp^\prime + 2\bigl(B(1 + \vp^\prime)(1 - \vp^\prime)^{-1} \vp^\prime \bigr)^{1/2},
\end{align*}
where the third line follows from Cauchy-Schwartz and the choice of $q_n < M_{n+1} - 2$. In summary, we have that 
\[
I \le B(1 + \vp_2) \|x\|^2
\]

Similarly, for $j\in I_+(n)$, $f_j = x_{b(j)} \in \oplus_{1 \le i \le N_n+1} H_i$ by \eqref{Een} and $j \le K_{N_n}$. Therefore, defining $H_1^n = \oplus_{q_n\le i \le N_n + 1} H_i$, we can apply \eqref{SewePlus}, \eqref{TienPlus} and Lemma \ref{L:project} to obtain
\begin{align*}
II&:=\sum_{n=1}^\infty \sum_{j\in I_+(n)} |\langle f_j, x\rangle|^2 \\
&\le \sum_{n=1}^\infty \biggl(B(1 + \vp^\prime) (1 - \vp^\prime)^{-1} \| P_{H_1^n} x\|^2 + \vp^\prime 2^{-n} + 2\bigl(B(1 + \vp^\prime)(1 - \vp^\prime)^{-1} \vp^\prime 2^{-n}\bigr)^{1/2} \| P_{H_1^n} x\|\biggr)\\
&\le B(1 + \vp^\prime)(1 + \vp^\prime)^{-1} + \vp^\prime + 2\bigr(B(1 + \vp^\prime)(1 - \vp^\prime)^{-1} \vp^\prime\bigr)^{1/2}\\
&\le B(1 + \vp_2),
\end{align*}
where the 3rd line follows from Cauchy-Schwartz and that $N_{n} + 1 < q_{n+1}$.

Therefore, we can conclude that
\begin{align*}
\sum_{n=1}^\infty \sum_{j\in I(n)} |\la f_j, x \ra|^2 &= I + II\\
&\le 2B(1 + \vp_2).
\end{align*}

 \end{proof}

\begin{proof}[Proof of Lower Bound] Let $x\in H$, $\|x\| = 1$. By \eqref{Ses} and Lemma \ref{L:block}, for each $n$ there exists $p_n \in (M_{n + 1}, N_n - 2]$ such that $\|P_{H_{p_n} \oplus H_{p_n + 1}} x\| \le \delta_n$, where $\delta_n$ were chosen to satisfy \eqref{Ses} and \eqref{Sewe}.  Let 
\[
y_1 = P_{\oplus_{j < p_1} H_j} x
\]
and for $n\ge 2$, let
\[
y_n =P_{\oplus_{p_{n-1}+1 < j < p_n} H_j} x.
\]

Note that
\begin{gather}
\|x - \sum_n y_n\| < \bigl(\sum_n \delta_n^2 \bigr)^{1/2} < \vp < \vp_2 \label{Elf}\\
1 - \vp_2 < 1 - \vp < \|\sum_n y_n\| \le 1 \label{Twallf}, \textrm{ and }\\
\forall y\in \oplus_{p_{n-1} + 1 < i < p_n} H_i, \,\, \sum_{j\in I(n)} |\la f_j, y \ra|^2 \ge A\|y\|^2\label{Dertien}\quad \textrm { by \eqref{Nege}. } 
\end{gather}
Next, we define
\begin{gather*}
I_{-1}(n) = \{j \in I(n): b(j) \le K_{p_{n-1}}\},\\
I_0(n) = \{j\in I(n) : K_{p_{n-1}} <b(j) \le K_{p_n} \}, \textrm{ and}\\
I_1(n) = \{j\in I(n) : b(j) > K_{p_n}\},
\end{gather*}
where $b$ is the map  defined in \eqref{DefB}.

We compute
\begin{align*}
\sum_{n=1}^\infty & \sum_{j\in I(n)} |\la f_j, \sum_{m=1}^\infty y_m \ra |^2 \ge \sum_{n=1}^\infty \sum_{j\in I_0(n)} |\la f_j, \sum_{m=1}^\infty y_m \ra|^2 \\
&= \sum_{n=1}^\infty \sum_{j\in I_0(n)} |\la f_j, \sum_{m=1}^{n+1} y_m \ra|^2  \\
&\ge \sum_{n=1}^\infty\biggl(\sum_{j\in I_0(n)} |\la f_j, y_n\ra|^2 - 2\bigl(B(1 + \vp^\prime)(1 - \vp^\prime)^{-1} \vp^\prime 2^{-n}\bigr)^{1/2} \|y_n\| \|\sum_{m \le n+1, m\not= n} y_m\| \biggr) \\
& > \biggl(\sum_{n=1}^\infty | \la f_j, y_n \ra |^2 \biggr) - \vp_2\\
&\ge \sum_{n=1}^\infty \biggl(\sum_{j\in I_0(n)} |\la f_j, y_n \ra|^2 + \bigl( \sum_{j\in I_1(n)} | \la f_j, y_n \ra|^2 - \vp^\prime 2^{-n}\|y_n\|^2\bigr) + \bigl(\sum_{j\in I_{-1}(n)} |\la f_j, y_n \ra |^2 - 0\bigr) \biggr) - \vp_2\\
&\ge \sum_{n=1}^\infty \bigl(A \|y_n\|^2 - \vp^\prime 2^{-n} \|y_n\|^2\bigr) - \vp_2\\
&\ge A\|\sum_{n=1}^\infty  y_n \|^2 - \vp^\prime - \vp_2.
\end{align*}
where the second line follows from $f_j \in \oplus_{k \le p_n + 1}H_j$ for $j\in I_0(n)$, the third line follows from Lemma \ref{L:project}, the fourth line follows from Cauchy-Schwartz and the definition of $\vp_2$, the fifth line follows from \eqref{SewePlus}, and the sixth line is from \eqref{Dertien}. We conclude the proof by using the above estimate for $\sum_j |\langle f_j,\sum_{m\in\N} y_m\rangle|^2$ to obtain a lower frame bound for  $(f_j)_{j\in I(n), n\in \N}$.

\begin{align*}
\sum_j |\langle f_j,x\rangle|^2&\ge \sum_j |\la f_j, \sum_{m\in \N} y_m \ra|^2 - 4B(1 + \vp_2) \|\sum_m y_m\|\, \| x - \sum_m y_m\|\\
&> A\|\sum_n y_n \|^2 - \vp^\prime - \vp_2 - 4B(1 + \vp_2)\vp_2\\
&=A(\|x\|^2-\|x-\sum y_n\|^2)-\vp^\prime - \vp_2 - 4B(1 + \vp_2)\vp_2\\
& > A \| x\|^2-A\vp_2^2-\vp^\prime - \vp_2 - 4B(1 + \vp_2)\vp_2 
\end{align*} 

The first line follows from Lemma \ref{L:project} and our Bessel bound for $(f_j)_{j\in I(n), n\in \N}$, the second line follows from \eqref{Elf} and the immediately preceding calculation, and the last line follows from \eqref{Elf}. Thus $(f_j)_{j\in I(n), n\in \N}$ has lower frame bound $A - A\vp^\prime-\vp_2^2 - \vp_2 - 4B(1 + \vp_2)\vp_2 $.

\end{proof}

\renewcommand{\qedsymbol}{$\square$}
This concludes the proof of Theorem \ref{T:samp}.
\end{proof}

\begin{thm}\label{T:disc}
Let $\Psi:X\rightarrow H$ be a continuous Parseval frame such that $\|\Psi(t)\|\leq 1$ for all $t\in X$.  If $A,B>0$ are the uniform constants given in Theorem \ref{T:reduction} then for all $\vp>0$ there exists $(s_i)_{i\in I}\in X^I$ such that $(\Psi(s_i))_{i \in I}$ is a frame of $H$ with lower frame bound $(1-\vp)A$ and upper frame bound $2(1+\vp)B$.
\end{thm}

\begin{proof}
 Without loss of generality we assume that $H$ is infinite dimensional.
 Let $1>\vp>0$.  Let $\vp_0>0$ such that $g(\vp_0) <\vp$, where $g$ is the function guaranteed to exist from Theorem \ref{T:samp}.   

By Lemma \ref{L:discretize} there exists a partition $(X_i)_{i\in\N}$ of $X$ and $(t_j)_{j\in\N}\subset X$ such that $t_j\in X_i$ for all $j\in\N$ and $\|\Psi(t)-\Psi(t_j)\|<\frac{\vp_0}{6}$ for all $t\in X_j$ and
$\int \left\| \Psi(t)-\sum_{i\in\N}\Psi(t_i) 1_{X_i}(t)\right\| d\mu(t)<\frac{\vp_0}{3}.
$  Let $\Phi:X\rightarrow H$ be given by $\Phi:=\sum \Psi(t_i) 1_{X_i}$.  We will prove that $\Phi$ is a continuous frame with upper frame bound $1+\vp_0$ and lower frame bound $1-\vp_0$.
Let $x\in H$ with $\|x\|=1$.  We first estimate the lower frame bound.
\begin{align*}
\|x\|^2&= \int |\langle x, \Psi(t)\rangle|^2 d\mu(t)\\
&= \int |\langle x, (\Psi(t)-\Phi(t)+\Phi(t))\rangle|^2 d\mu(t)\\
&\leq \int |\langle x, \Phi(t)\rangle|^2+2|\langle x, (\Psi(t)-\Phi(t))\rangle||\langle x,\Phi(t)\rangle|+|\langle x,(\Psi(t)-\Phi(t))\rangle|^2 d\mu(t)\\
&\leq \int |\langle x, \Phi(t)\rangle|^2+2\|\Psi(t)-\Phi(t)\|\|\Phi(t)\|+\|\Psi(t)-\Phi(t)\|^2 d\mu(t)\\
&\leq \int |\langle x, \Phi(t)\rangle|^2+2\|\Psi(t)-\Phi(t)\|+\|\Psi(t)-\Phi(t)\|\frac{\vp_0}{3} d\mu(t) \\
&\qquad \textrm{ as }\|\Phi(t)\|\!\leq\! 1\textrm{ and }\|\Psi(t)-\Phi(t)\|\!<\!\frac{\vp_0}{3} \\
&\leq \int |\langle x, \Phi(t)\rangle|^2 d\mu(t)+2(\frac{\vp_0}{3}) +(\frac{\vp_0}{3})^2\qquad \bigg( \textrm{as }\int \| \Psi(t)-\Phi(t)\| d\mu(t)< \frac{\vp_0}{3} \biggr)
\\
&<\int |\langle x, \Phi(t)\rangle|^2 d\mu(t)+\vp_0
\end{align*}
Thus, $\Phi$ has lower frame bound  $1-\vp_0$.  We similarly estimate the upper frame bound.
\begin{align*}
\|x\|^2&= \int |\langle x, \Psi(t)\rangle|^2 d\mu(t)\\
&= \int |\langle x, (\Psi(t)-\Phi(t)+\Phi(t))\rangle|^2 d\mu(t)\\
&\geq \int |\langle x, \Phi(t)\rangle|^2-2|\langle x, (\Psi(t)-\Phi(t))\rangle||\langle x,\Phi(t)\rangle|+|\langle x,(\Psi(t)-\Phi(t))\rangle|^2 d\mu(t)\\
&> \int |\langle x, \Phi(t)\rangle|^2 d\mu(t)-2(\frac{\vp_0}{3}) >\int |\langle x, \Phi(t)\rangle|^2 d\mu(t)-\vp_0\\
\end{align*}
Thus, $\Phi$ has upper frame bound $1+\vp_0$.

We have that for all $x\in H$, $\int |\langle x, \Phi(t)\rangle|^2 d\mu(t)=\sum |\langle x, \Psi(t_i)\rangle|^2\mu(X_i)$.
  Thus, $(\sqrt{\mu(X_i)}\Psi(t_i))_{i\in\N}$ is a frame of $H$ with lower frame bound $1-\vp_0$ and upper frame bound
  $1+\vp_0$.   
  Therefore, by Theorem \ref{T:samp}, there exists a sequence of natural numbers $I$ such that 
  $
 (\Psi(t_i))_{i\in I}
    $
    is a frame for $H$ with bounds $A(1 - \vp)$ and $2B(1 + \vp)$, as desired.
  \end{proof}

We now show that the complete solution to the discretization problem can be reduced to the special  case of Theorem \ref{T:disc}.  We restate and prove Theorem \ref{T:D} that was presented in the introduction.

\begin{thm}
Let $(X,\Sigma)$ be a measure space such that every singleton is measurable and let $\Psi:X\rightarrow H$ be measurable.  There exists $(t_j)_{j\in J}\in X^J$ such that $(\Psi(t_j))_{j \in J}$
 is a frame of $H$ if and only if there exists a positive, $\sigma$-finite measure $\nu$ on $(X,\Sigma)$ so that  $\Psi$ is a continuous frame of $H$ with respect to $\nu$ which is bounded $\nu$-almost everywhere.  
\end{thm}

\begin{proof}
 We first assume that there exists a positive $\sigma$-finite measure $\nu$ on $X$ so that  $\Psi:X\rightarrow H$ is a continuous frame of $H$ with respect to $\nu$ which is bounded almost everywhere.  By passing to a measurable subset of $X$ of full measure we may assume that $\Psi$ is bounded.  Let $T:H\rightarrow H$ be the frame operator of $\Psi$.  Note that $T$ is a positive self-adjoint invertible operator.  Define $\Phi:X\rightarrow H$ by $\Phi(x)= T^{-1/2} \Psi(x)$.
Thus, $\Phi$ is a continuous Parseval frame by Lemma \ref{L:parseval} and $\Phi$ is bounded as it is the composition of bounded functions.  There exsists $C>0$ such that $\|\Phi(t)\|\leq C$ for all $t\in X$.  We now define a measure $\nu_0$ on $X$ by $\nu_0=C^2 \nu$.  Then, $C^{-1}\Phi:X\rightarrow H$ is a continuous Parseval frame with respect to $\nu_0$ such that $\|C^{-1}\Phi(t)\|\leq 1$ for all $t\in H$.  Thus, there exists $(t_j)_{j\in J}\in X^J$ such that $(C^{-1}\Phi(t_j))_{j\in J}$ is a frame of $H$ by Theorem \ref{T:disc}.  We have that $(C^{-1} T^{-1/2} \Psi(t_j))_{j\in J}$ is a frame of $H$ and hence $( \Psi(t_j))_{j\in J}$ is a frame of $H$.

We now assume that there exists $(t_j)_{j\in J}\in X^J$ such that $(\Psi(t_j))_{j \in J}$
 is a frame of $H$. Note that $(\Psi(t_j))_{j \in J}$ is bounded as it is a frame.
  It is possible for some points to be sampled multiple times.
    For $t\in \cup_{j\in J} t_j$, we let $n_t=|\{j\in J\,:\,t=t_j\}|$.  We define a measure $\nu$ on $X$ by $\nu(A)=\sum_{t\in A\cap  \cup_{j\in J} t_j }n_t$.
    As singletons are measurable, we have that $\nu$ is a measure on $(X,\Sigma)$ and the frame operator for the map $\Psi:\rightarrow H$ with respect to $\nu$ is the same as the frame operator of
$(\Psi(t_j))_{j \in J}$.  Thus, $\Psi$ is a continuous frame with respect to $\nu$ which is bounded on $\cup_{j\in J} t_j$ and  $\nu(X\setminus \cup_{j\in J} t_j)=0$.
\end{proof}

\end{document}